\documentclass[10pt, reqno]{amsart}
\usepackage{amsmath,amsthm,amsfonts,amssymb,mathrsfs,bm,graphicx,stmaryrd,caption}
\usepackage{subcaption}
\usepackage{epstopdf}
\usepackage{mathtools}
\usepackage{dsfont}
\usepackage{multicol}
\usepackage{adjustbox}
\usepackage[colorlinks=true,linkcolor=blue,citecolor=blue,breaklinks]{hyperref}
\usepackage{url}

\usepackage{breakurl}
\usepackage{bbm}
\usepackage{upgreek}
\newcommand{\scT}{\mathscr{T}}

\newcommand{\LL}{\mathbb{L}}

\newcommand{\cR}{\mathcal{R}}

\newcommand{\RR}{\mathbb{R}}

\newcommand{\hor}{\mathrm{hor}}

\newcommand{\PP}{\mathbb{P}}

\newcommand{\fp}{\mathfrak{p}}
\newcommand{\TT}{\mathbb{T}}

\newcommand{\cT}{\mathcal{T}}

\newcommand{\EE}{\mathbb{E}}

\newcommand{\fluc}{\mathrm{Fluc}}
\newcommand{\NN}{\mathbb{N}}

\newcommand{\ZZ}{\mathbb{Z}}

\newcommand{\fq}{\mathfrak{q}}

\newcommand{\sqr}{\mathrm{Square}}

\usepackage[letterpaper,hmargin=1.0in,vmargin=1.0in]{geometry}
\parskip 	\smallskipamount

\newtheorem{theorem}{Theorem}
\newtheorem{lemma}[theorem]{Lemma}
\newtheorem{conjecture}[theorem]{Conjecture}
\newtheorem{question}[theorem]{Question}

\newtheorem{proposition}[theorem]{Proposition}

\theoremstyle{definition}

\newcommand{\ind}{\mathbbm{1}}
\newcommand{\wgt}{\mathrm{Wgt}}

\def\Var{{\rm Var}}

\newcommand{\cO}{\mathcal{O}}

\newcommand{\Cov}{\mathrm{Cov}}

\newcommand{\boxx}{\mathrm{Box}}

\newcommand{\slope}{\mathrm{slope}}

\newcommand{\0}{\mathbf{0}}
\newcommand{\n}{\mathbf{n}}

\usepackage[backend=biber,style=alphabetic,doi=false,maxalphanames=10,maxnames=50]{biblatex}
\AtBeginBibliography{\small}

\begin{document}
\title[]{Near-existence of bigeodesics in dynamical exponential last passage percolation}%

\author[]{Manan Bhatia}
\address{Manan Bhatia, Department of Mathematics, Massachusetts Institute of Technology, Cambridge, MA, USA}
\email{mananb@mit.edu}
\date{}
\makeatletter
\let\thefootnote\relax
\footnotetext{This preprint is one of two works that together replace the earlier preprint [arXiv:2504.12293v1]. The companion article \cite{Bha25+} analyses geodesic switches and exceptional times in dynamical Brownian last passage percolation.}
\begin{abstract}
  It is believed that, under very general conditions, bi-infinite geodesics (or bigeodesics) do not exist for planar first and last passage percolation (LPP) models. However, if one endows the model with a natural dynamics, thereby gradually perturbing the geometry, then it is plausible that there could exist a non-trivial set $\scT$ of exceptional times at which such bigeodesics exist. For dynamical exponential LPP, we show that $\scT$ is ``very close'' to being non-trivial; namely, we obtain an $\Omega( 1/\log n)$ lower bound on the probability that there exists a random time $t\in [0,1]$ at which a non-trivial geodesic of length $n$ passes through the origin at its midpoint; note that if the above probability were $\Omega(1)$, then it would imply the non-triviality of $\scT$. We conjecture that, even if $\scT\neq \emptyset$, it a.s.\ has Hausdorff dimension exactly zero. 
\end{abstract}
\maketitle
\tableofcontents
\section{Introduction}
\label{sec:intro}

First passage percolation (FPP) is a natural lattice model of random geometry where the Euclidean metric is distorted by i.i.d.\ noise. To define it on the planar lattice, one considers a family of i.i.d.\ random variables $\omega=\{\omega_{\{x,y\}}\}_{x\sim y\in \ZZ^2}$, where $x\sim y$ denotes adjacency in $\ZZ^2$ and simply defines the length of a lattice path $\gamma$ as the total weight of all the edges it utilises. Thereafter, the distance between points is defined as the infimum over the length of all lattice paths between the two points, and any path which attains this minimum is called a geodesic. While the above model is simple to define, it exhibits rich mathematical structure-- indeed, it is believed that for a wide class of weight distributions of the vertex weights, planar FPP is in the Kardar-Parisi-Zhang (KPZ) \cite{KPZ86} universality class, which is a class of random growth models that are expected to share the same universal behaviour.

While geodesics between any two points always exist, a long-standing question for FPP is whether any ``bigeodesics'' exist, where the latter refers to a bi-infinite lattice path whose every finite segment is a geodesic. The first reference to the above question appears to be \cite{Kes86}, where it is attributed to Furstenberg. It is believed that under very mild restrictions on the weight distribution, bigeodesics a.s.\ do not exist in FPP. In fact, the question of bigeodesics in FPP can be formulated in a completely different context-- that of the disordered Ising ferromagnet. The latter can be formally defined as the planar statistical physics model corresponding to the Hamiltonian $H(\sigma)=-\sum_{x\sim y\in \ZZ^2}\eta_{\{x,y\}}\sigma_x\sigma_y$, where $\eta=\{\eta_{\{x,y\}}\}_{x\sim y}$ is an i.i.d.\ positive noise field and $\sigma_x\in \{+1,-1\}$ for all $x\in \ZZ^2$. It turns out that there is a direct correspondence turning an instance of FPP into an instance of the latter and in this correspondence, bigeodesics correspond to non-constant ground states, where we note that a $\sigma$ which is globally constant is trivially a ground state since $\eta$ is positive. In this context, the conjecture states that for very general coupling constant distributions, there a.s.\ do not exist any non-constant ground states for the disordered Ising ferromagnet. Unfortunately, as is the case for most questions regarding FPP, the above question remains open. %
We refer the reader to the survey \cite{ADH17} for a discussion on the question of the non-existence of bigeodesics and of the connection to the disordered Ising ferromagnet. We note that while this paper is concerned with the planar case, it is also possible to consider high dimensional versions of the question discussed above, where one now looks at the disordered Ising ferromagnet on $\ZZ^d$ for $d\geq 3$ and the corresponding FPP question now involves ``minimal surfaces'' instead of bigeodesics and recently, there have been several interesting works in this direction \cite{BGP23,DEP24,DG23}.

Instead of studying the `static' model of planar FPP, one might wonder what happens if one dynamically evolves the noise $\omega=\{\omega_{\{x,y\}}\}_{x\sim y\in \ZZ^2}$-- how does this evolve the associated random geometry? For instance, a natural dynamics is to simply consider the noise field $\omega^t=\{\omega^t_{\{x,y\}}\}_{x\sim y\in \ZZ^2}$ obtained by updating the noise field $\omega$ via independent resampling according to independent exponential clocks associated to each vertex. For the corresponding planar disordered Ising ferromagnet, the above corresponds to an independent resampling dynamics of the coupling constant field $\eta$. Assuming that the conjecture from the previous paragraph on the non-existence of bigeodesics in static FPP is indeed true, one might ask the following question, and this question is the guiding force behind this work.

\begin{question}
  \label{que:1.1}
Does dynamical first passage percolation have any exceptional times at which bigeodesics exist? Equivalently, does the dynamical disordered Ising ferromagnet possess any exceptional times at which non-constant ground states exist? If such exceptional times do exist, how frequently do they occur, as measured by their Hausdorff dimension? %
\end{question}

For us, an important motivation for considering the above question is the analogous study of noise sensitivity \cite{BKS99} and exceptional phenomena in the context of critical percolation, a classical and very well-studied model. From the work of \cite{Har60,Kes80}, it has long been known that at criticality, for critical percolation on the square lattice, there a.s.\ does not exist any infinite cluster. In an exciting sequence of works \cite{SS10, GPS10, GPS18}, a quantitative study of noise-sensitivity in critical percolation was carried out, where noise-sensitivity refers to the phenomenon in which the resampling of a microscopic amount of noise leads to a macroscopic change in the connectivity properties. Using this, it was further established that for dynamical critical site percolation on the triangular lattice, there exist exceptional times at which a giant cluster exists and that the set of such exceptional times a.s.\ has Hausdorff dimension $31/36$. We refer the reader to \cite{GS15} for an exposition discussing the above line of research and the background discrete Fourier analytic techniques.

Now, even for FPP, it is expected that one has noise-sensitivity in the sense that resampling a small number of edge weights should lead to a macroscopic change in the geodesic structure-- while this has not been shown for lattice FPP, such results have been established for the closely related last passage percolation models for which much more is now known \cite{Cha14,GH24,ADS24}, and which we shall shortly discuss. Thus, in view of the above expected noise-sensitivity of FPP, one might wonder whether bigeodesics, which are not expected to exist for static FPP can in fact exist at some exceptional times in dynamical FPP, and this is the content of Question \ref{que:1.1}.

However, considering that even the more basic question of the non-existence of bigeodesics has not been answered yet for static FPP, Question \ref{que:1.1}, as stated, does not seem tractable at the moment. Thus, in this paper, we in fact do not work with FPP but instead consider Question \ref{que:1.1} in the context of integrable last passage percolation models, and the most canonical such model is exponential last passage percolation, which we now define. Let $\{\omega_{z}\}_{z\in \ZZ^2}$ be a field of i.i.d.\ $\exp(1)$ random variables. Now, for any two points $p\leq q\in \ZZ^2$, by which we mean that the inequality holds coordinate-wise, and any lattice path $\gamma$ from $p$ to $q$ which takes only up and right steps, we define $\wgt(\gamma)=\sum_{z\in \gamma}\omega_z$ and finally, we define the last passage time $T_p^q=\max_{\gamma: p\rightarrow q}\wgt(\gamma)$, where the maximum is over all ``up-right'' paths $\gamma$ from $p$ to $q$. Further, there is a.s.\ a unique path $\gamma$ attaining the above maximum, and this path is called the geodesic from $p$ to $q$ and is denoted by $\Gamma_p^q$. %

While FPP is believed to be in the KPZ universality class, exponential LPP is known to be so and in particular, it is expected that, under mild restrictions on the weight distribution, the scaling limit of FPP is the also the directed landscape \cite{DOV18}, the scaling limit of exponential LPP \cite{DV21}. In contrast with FPP, much is known about exponential LPP owing to its integrability. For instance, in the case of exponential LPP, the question of the non-existence of bigeodesics has been settled. Indeed, it was established in \cite{BHS22} that non-trivial bigeodesics do not exist in exponential LPP, where trivial bigeodesics refer to lattice paths which are either completely vertical or completely horizontal; note that such paths are always bigeodesics due to the directed nature of LPP and thus it is only interesting to consider non-trivial bigeodesics. As a result of the above, if we consider a dynamical version of exponential LPP, where we have an independent exponential clock at each vertex $z\in \ZZ^2$, and we simply independently resample the weight $\omega_z$ when the clock at $z$ rings, then we can consider the following analogue of Question \ref{que:1.1}.

\begin{question}
  \label{que:1}
  Does dynamical exponential LPP have exceptional times $t$ at which non-trivial bigeodesics exist? If so, what is the Hausdorff dimension of this set?
\end{question}

While working with dynamical exponential LPP, we shall use $T^t$ to denote the LPP at time $t\in \RR$ and for any points $p\leq q\in \ZZ^2$, we shall use $T_p^{q,t}$ to denote the passage time from $p$ to $q$ for the LPP at time $t$. Further, we shall use $\Gamma_p^{q,t}$ to denote the geodesic from $p$ to $q$; it is not difficult to see that $\Gamma_p^{q,t}$ is a.s.\ unique simultaneously for all $t\in \RR$ and all $p\leq q\in \ZZ^2$.

\subsection{Main result}
\label{sec:main-results}
We are now ready to state the main result of this work. While we have not been able to resolve Question \ref{que:1} in this work, we show that ``exceptional times are very close to existing''-- namely, we prove the following quantitative subpolynomial bound on the presence of times admitting unusually long geodesics. In the following and throughout the paper, we use the functions $\phi(x,y)=x+y,\psi(x,y)=x-y$ for $x,y\in \RR$; also, we shall often use $\0,\n$ to denote the points $(0,0),(n,n)\in \ZZ^2$.

\begin{figure}
  \centering
  \includegraphics[width=0.4\linewidth]{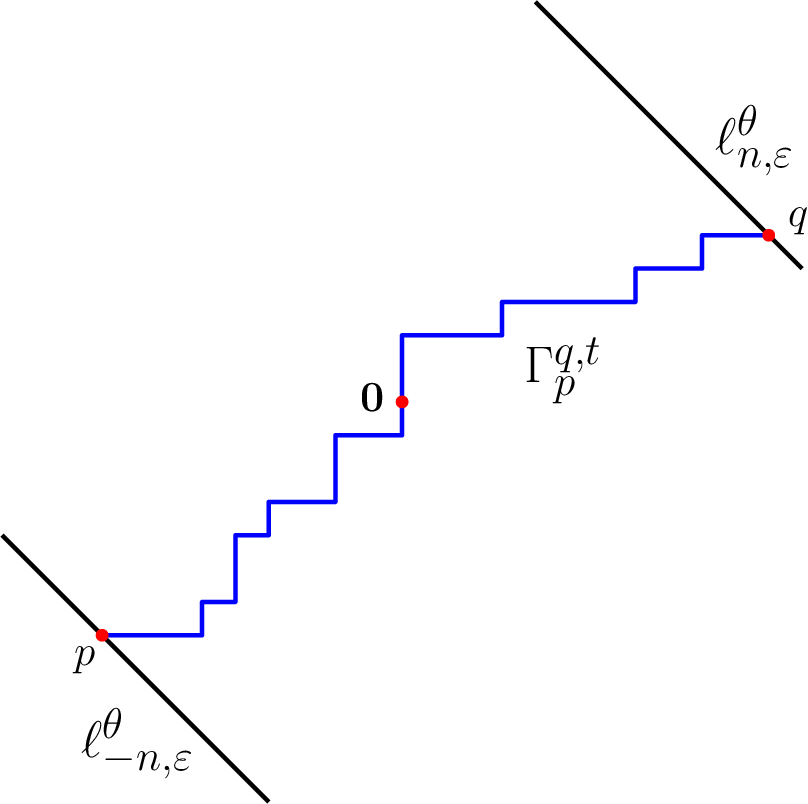}
  \caption{Statement of Theorem \ref{thm:4}: There is at least a $C(\log n)^{-1}$ probability of there existing a $t\in [0,1]$ for which there is a geodesic $\Gamma_p^{q,t}$ between two points $p,q$ on the linear length segments $\ell_{-n,\varepsilon}^\theta, \ell_{n,\varepsilon}^\theta$ which additionally satisfies $\0\in \Gamma_p^{q,t}$.}
  \label{fig:lb-thm}
\end{figure}
\begin{theorem}
  \label{thm:4}
  Consider dynamical exponential LPP and fix $\theta\in (-1,1),\varepsilon>0$. For $n\in \ZZ$, let $\ell_{n,\varepsilon}^\theta$ denote the line segment defined by $\ell_{n,\varepsilon}^\theta=\{p: \phi(p)=n, |\psi(p)-\theta n|\leq \varepsilon |n|\}$. Then there exists a constant $C$ such that for all large enough $n\in \NN$, we have
  \begin{equation}
    \label{eq:336}
    \PP(\exists t\in [0,1]\textrm{ and points } p\in \ell_{-n,\varepsilon}^\theta, q\in \ell_{n,\varepsilon}^\theta \textrm { with } \0\in \Gamma_{p}^{q,t})\geq C(\log n)^{-1}.
  \end{equation}
\end{theorem}
We refer the reader to Figure \ref{fig:lb-thm}. The crucial aspect of Theorem \ref{thm:4} is that the factor $(\log n)^{-1}$ decays subpolynomially in $n$ and thus the bound above is better than any lower bound of the form $n^{-\alpha+o(1)}$. We note that if one upgrades the above $(\log n)^{-1}$ lower bound to one that does not decay with $n$, then it is plausible that this (combined with an ergodicity argument) would answer Question \ref{que:1} in the affirmative. In fact, even if bigeodesics as in Question \ref{que:1} do exist, we expect that they ``barely'' do so, and this is made precise by the following conjecture.%
\begin{conjecture}
  \label{conj:1}
Fix a dynamical LPP model and consider the set of times $\scT$ at which non-trivial bigeodesics exist. Then the set $\scT$ almost surely has Hausdorff dimension $0$.
\end{conjecture}

In fact, as we shall see in the heuristic argument for the above conjecture in Section \ref{sec:are-upper-bounds}, we expect the value $0$ to be ``tight'' in the sense that for any fixed $K>0$, we expect to have
\begin{equation}
  \label{eq:337}
  \PP(\exists t\in [0,\varepsilon], \theta \in [K^{-1},K]\textrm{ and a }  \theta \textrm{-directed bigeodesic } \Gamma^{t}\ni \0)=\varepsilon^{1-o(1)},
\end{equation}
where we say that an unbounded set $A\subseteq \RR^2$ is $\theta$-directed if for any sequence $(x_n,y_n)\in A$ with $|y_n|\rightarrow \infty$, we have $\lim_{n\rightarrow \infty}x_n/y_n=\theta$.

We note that there are quite a few settings where the first moment analysis of an exceptional set yields an exponent corresponding precisely to dimension $0$-- in such settings it is usually very difficult to prove/disprove the existence of such exceptional points or even prove any quantitative estimates about such points, since one can no longer ignore subpolynomial errors.

  \paragraph{\textbf{Notational comments}} For $p\neq q\in \RR^2$, we use $\LL_p^q$ to denote the line joining $p$ and $q$. Frequently, for $a<b\in \RR$, we shall work with the discrete intervals $[\![a,b]\!]=[a,b]\cap \ZZ$. Often, we shall use the boldface letters $\0,\mathbf{m},\mathbf{n}$ to denote $(0,0),(m,m),(n,n)\in \RR^2$. For points $p=(x_1,y_1),q=(x_2,y_2)\in \RR^2$, we shall write $p\leq q$ if $x_1\leq x_2$ and $y_1\leq y_2$. For points $p=(x_1,y_1)\neq (x_2,y_2)\in \RR^2$, we define $\slope(p,q)=\frac{x_2-x_1}{y_2-y_1}$: note that this is the inverse of the usual definition of the slope of a line. For a finite set $A\subseteq \ZZ^2$, we shall use $|A|$ to denote the cardinality of $A$. For $\theta\in \RR$, an unbounded set $A\subseteq \RR^2$ is said to be $\theta$-directed if for any sequence $(x_n,y_n)\in A$ with $|y_n|\rightarrow \infty$, we have $\lim_{n\rightarrow \infty}x_n/y_n=\theta$.

\paragraph{\textbf{Acknowledgements}} We thank Riddhipratim Basu for the discussions. The author acknowledges the partial support of the NSF grant DMS-2153742 and the MathWorks fellowship.
\section{Model definitions and background}
\label{sec:prelim}
\subsection{Dynamical Exponential last passage percolation}
\label{sec:model-exp}
We start with $\ZZ^2$ endowed with a field $\{\omega_{z}\}_{z\in \ZZ^2}$ of i.i.d.\ $\exp(1)$ random variables. Now, for each up-right path $\gamma$ between points $p\leq q\in \ZZ^2$, we define
\begin{equation}
  \label{eq:334}
  \wgt(\gamma)=\sum_{z\in \gamma}\omega_z,
\end{equation}
We define the last passage time
\begin{equation}
  \label{eq:335}
  T_p^q=\max_{\gamma: p\rightarrow q}\wgt(\gamma),
\end{equation}
where the maximum is over all up-right paths $\gamma$ from $p$ to $q$. It is easy to see that almost surely, for any $p\leq q$ as above, there exists a unique path $\Gamma_p^q$ attaining the above maximum and this path is called the geodesic from $p$ to $q$. Sometimes, we shall think of the geodesic $\Gamma_p^q$ as a function, that is, for $r\in [\![\phi(p),\phi(q)]\!]$, we shall use $\Gamma_p^q(r)$ to denote $\psi(z)$, where $z$ is the unique point in $\ZZ^2$ with $\phi(z)=r$ and $z\in \Gamma_p^q$. This completes the discussion of the model of (static) exponential LPP, and we now describe the dynamics that we shall work with.

We start with the field $\omega^0=\{\omega^0_z\}_{z\in \ZZ^2}\stackrel{d}{=}\omega$. Now, we attach an independent exponential clock of rate $1$ to each vertex $z\in \ZZ^2$; if the clock corresponding to $z$ rings at time $t$, then we independently resample the value $\omega^t_z$. The above defines the process $\{\omega^t_z\}_{z\in \ZZ^2,t\geq 0}$, and we note that this is stationary in $t$. Finally, by using a Kolmogorov extension argument, we can extend the above definition to obtain the process $\{\omega^t_z\}_{z\in \ZZ^2,t\in \RR}$.

With the above at hand, we define the LPP $T^t$ by using the definition \eqref{eq:335} with the environment $\omega$ now replaced by $\omega^t$. We can correspondingly define the geodesics $\Gamma_{p}^{q,t}$ for points $p\leq q\in \ZZ^2$ associated to the LPP $T^t$.

\subsection{Bigeodesics and their non-existence in static exponential LPP}
\label{sec:bigeod}
Recall that in static exponential LPP, geodesics a.s.\ exist between any two points. As opposed to finite geodesics which are analogues of line segments in Euclidean geometry, one might also consider bi-infinite geodesics, or more simply, bigeodesics. In exponential LPP, a bigeodesic is a bi-infinite up-right path $\gamma$ such that every segment of it is a finite geodesic. Note that any entirely horizontal or vertical bi-infinite lattice path (resp.\ staircase) is trivially a bigeodesic. It is believed that under very general conditions, non-trivial bigeodesics a.s.\ do not exist in last passage percolation models. For exponential LPP, this was proved in \cite{BHS22} (see also \cite{BSS20} for another proof).
\begin{proposition}[{\cite[Theorem 1]{BHS22}, \cite[Theorem 1.1]{BBS20}}]
  \label{prop:27}
  Almost surely, there exist no non-trivial bigeodesics in static exponential LPP.
\end{proposition}
Before moving on, we remark that the work \cite{Ale23} shows that, in the setting of planar FPP, under certain strong unproven assumptions, bigeodesics a.s.\ do not exist (see also \cite{DEP24+}).

\subsection{Motivation: exceptional times in dynamical percolation}
\label{sec:motiv-perc}
An important motivation for this paper is the advancement in the understanding of dynamical percolation in the past two decades, and we now briefly discuss this. Prompted by a question of Paul Malliavin during a seminar at the Mittag-Leffler institute in the spring of 1995, the model of dynamical percolation was introduced in the work \cite{HPS97}; we now briefly introduce a version of this model. Consider the triangular lattice $\TT\subseteq \RR^2$ with the usual graph structure and consider critical Bernoulli site percolation on this lattice. That is, we have a field of i.i.d.\ $\mathrm{Ber}(1/2)$ variables $\{\omega_{v}\}_{v\in \TT}$, and we think of vertices with $\omega_{v}=1$ as open and the others as closed. The central question of interest in this model is whether there exists an infinite connected cluster of vertices $v$ which are all open.
\begin{proposition}[{\cite{Kes80}}]
  \label{prop:28}
  Consider critical Bernoulli site percolation on $\TT$. Almost surely, there does not exist any infinite open cluster.
\end{proposition}
Now, one can also define a dynamical version of the above model. Indeed, we could just define $\omega^0\stackrel{d}{=}\omega$ and further have i.i.d.\ exponential clocks for all $v\in \TT$; if a clock rings at (say) time $t$, then we just independently resample the value $\omega^t_v$. This leads to a stationary process $\{\omega^t\}_{t\geq 0}=\{\omega^t_v\}_{v\in \TT,t\geq 0}$ which can be extended to a stationary process $\{\omega^t\}_{t\in \RR}=\{\omega^t_v\}_{v\in \TT,t\in \RR}$. As a result, one can now look at the percolations given by $\omega^t$ simultaneously for all $t\in \RR$.

In the works \cite{BKS99, SS10, GPS10}, it was shown that the above critical percolation model is in fact ``noise-sensitive'' in the sense that small perturbations to the system can lead to measurable changes in the macroscopic behaviour of the system. A refined understanding of this noise sensitivity behaviour led to the following remarkable statement about the behaviour of dynamical percolation.

\begin{proposition}[{\cite[Theorem 1.4]{GPS10}}]
  \label{prop:29}
 Let $\scT$ denote the set of times $t\in \RR$ such that the $\omega^t$ has an infinite open cluster. Then almost surely, the set $\scT$ is non-empty and has Hausdorff dimension $31/36$.
\end{proposition}
The above result is an important motivation for this work. We note that in Propositions \ref{prop:28}, \ref{prop:29}, a certain structure (presence of an infinite open cluster) is a.s.\ not present in the static model, but in fact, is a.s.\ present for the dynamical version at a random non-empty exceptional set of times. The goal of this work is to initiate a corresponding study of exceptional times in dynamical last passage percolation models with the statistic of interest being the presence of bigeodesics. As we saw in Proposition \ref{prop:27}, these are known to a.s.\ not exist in static exponential LPP, and in this work, we investigate the exceptional set of times at which bigeodesics exist in dynamical LPP.

Before moving on, we note that there have recently been some works investigating the presence of exceptional times at which there is a change in the limit shape for the model of dynamical ``critical'' FPP; we note that for static standard FPP, a macroscopic deviation from the limit shape is known to be superpolynomially rare and using this, such exceptional times are known to not exist \cite{Ahl15}. However, in critical FPP, where the probability of the weight of an edge being zero is equal to the critical probability for Bernoulli bond percolation, such exceptional times have been shown to exist for some regimes \cite{DHHL23,DHHL23+}.  We note that the behaviour of planar critical FPP is very different from standard (subcritical FPP) and LPP, and in particular, planar critical FPP is not in the KPZ universality class.

\subsection{Noise-sensitivity and the $n^{-1/3}$ time scale for the onset of chaos in LPP}
\label{sec:chaos-onset}
As discussed in Section \ref{sec:motiv-perc}, a crucial property of critical percolation is its noise-sensitivity. At an intuitive level, for any dynamics, noise sensitivity is directly linked to the presence of exceptional times. Indeed, heuristically, the more noise-sensitive a model is, the more ``independent'' chances it has to exhibit the exceptional configuration as the dynamics proceeds making it more ``likely'' for exceptional times to exist. As a result, for Question \ref{que:1}, it is imperative to investigate the noise-sensitivity properties of LPP.

In fact, there have been significant advances with regard to the above in recent years. First, the concept of noise-sensitivity and its connection to the notation of influences from the analysis of Boolean functions was introduced in \cite{BKS99}-- this machinery was later used in \cite{BKS03} to obtain an $O(n/\log n)$ bound for the variance of distances in FPP in the special case where the edge distribution is supported on only two values. Later, in the work \cite{Cha14}, a correspondence of superconcentration and chaos was discussed for various models, with one of them being LPP with Gaussian weights. In particular, with the help of a certain ``dynamical formula'' that holds for Gaussian LPP, it was shown that for many models, a statistic is superconcentrated in the static case if and only if it is noise-sensitive in the dynamical version. More recently, a finer study \cite{GH24} of noise-sensitivity was done in the context of dynamical LPP. Here, it was shown that the correct time scale at which chaos manifests in the setting of Brownian LPP with the Ornstein-Uhlenbeck dynamics is $n^{-1/3}$, and we now provide a statement for the above without formally defining Brownian LPP (see \cite[Sections 1.1,1.2]{GH24} for a definition). Locally, we shall define $\pi_1(x,y)=(x,0)$ and for a bounded set $A\subseteq \RR^2$, we shall use $|A|_{\hor}=\mathrm{Leb}(\pi_1(A))$, where $\mathrm{Leb}$ denotes the one dimensional Lebesgue measure.

\begin{proposition}[{\cite[Theorem 1.3]{GH24}}]
  \label{prop:44}
  Consider dynamical Brownian LPP defined using a family of independent Brownian motions $\{W_n^t\}_{n\in \ZZ,t\in \RR}$ independently evolving according to the Ornstein-Uhlenbeck dynamics. Use $\widetilde{\Gamma}_{\0}^{\n,t}$  to denote a geodesic from $\0$ to $\n$ at time $t$. Consider the quantity $\cO_n(t)= |\widetilde{\Gamma}_{\0}^{\n,0}\cap \widetilde{\Gamma}_{\0}^{\n,t}|_{\hor}$. Then for any fixed $\delta>0$, and for all $n$ large enough, we have
  \begin{align}
    \label{eq:520}
    &\EE \cO_n(t)= \Theta(n) \textrm{ for all } t<n^{-1/3-\delta},\nonumber\\
    &\EE \cO_n(t)=o(n) \textrm{ for all } t>n^{-1/3+\delta}.
  \end{align}
\end{proposition}
While, strictly speaking, the above result is not used in the proofs of the rigorous results of this paper, the $n^{-1/3}$ time scale for chaos is important at an intuitive level for this work. Indeed, it shall feature in Section \ref{sec:are-upper-bounds}, where we discuss the reasoning behind Conjecture \ref{conj:1}.

While Proposition \ref{prop:44} is in the setting of Brownian LPP, there have recently been works obtaining partial versions of the $n^{-1/3}$ time scale of chaos for FPP under certain assumptions \cite{ADS23} and for exponential LPP \cite{ADS24}. In particular, in \cite{ADS24}, a version of the dynamical formula from \cite{Cha14} is obtained for exponential LPP-- a slight generalisation of this will be important for this paper and we now discuss this.

\subsection{Dynamical Russo-Margulis formula}
\label{sec:stab-chaos-dynam}
We now discuss the above-mentioned dynamical formula relating geodesic overlaps in exponential LPP with covariances of passage times. Similar formulae have also been recently used for dynamical critical percolation \cite{TV23}, and as explained therein, such formulae can be considered to be a dynamical version of the classical Russo-Margulis formula from static Bernoulli percolation.

Following the notation in \cite{ADS23}, for a CDF $F$, let $X=\{X(i)\}_{i=1}^m$ be i.i.d.\ samples drawn from $F$ and for $r\in [0,1]$, let $Y_r=\{Y_r(i)\}_{i=1}^m$ be variables obtained by resampling each $X(i)$ independently with probability $r$ each. For $i\in [\![1,m]\!]$, let $\sigma_i^x\colon \RR^m\rightarrow \RR$ be the function which simply replaces the $i$th coordinate by $x$ and leaves the remaining coordinates unchanged. Now, let $f\colon \RR^m\rightarrow \RR$ be a function and for $i\in [\![1,m]\!]$ and $x\in \RR$, consider the operator $D_i^x$ which is defined by
\begin{equation}
  \label{eq:657}
  D_i^xf= f\circ \sigma_i^x - \int dF(y) f\circ \sigma_i^y .
\end{equation}
Using the above, for functions $f,g\colon \RR^m\rightarrow \RR$, we define the co-influence of index $i$ with respect to $f$ and $g$ at times $0$ and $r$ by
\begin{equation}
  \label{eq:55}
  \mathrm{Inf}_i^{f,g}(r)=\int \EE[D_i^xf(Y_0)D_i^xg(Y_r)]dF(x).
\end{equation}
The following result on the derivative of the co-influences can be obtained by a slight modification of the proof of \cite[Proposition 6]{ADS23}.
\begin{proposition}
  \label{prop:10.1}
  For functions $f,g\colon \RR^m\rightarrow \RR$ satisfying $\EE [f(X)^2], \EE [g(X)^2]<\infty$, and any $r\in (0,1)$, we have
  \begin{equation}
    \label{eq:34}
    \frac{d}{dr}\EE[f(Y_0)g(Y_r)]=-\sum_{i=1}^m \mathrm{Inf}_i^{f,g}(r).
  \end{equation}
\end{proposition}
We note that in \cite{ADS23}, the above result is proved for the case $g=f$. However, the proof therein works verbatim to yield the above more general statement as well. Recalling that exponential LPP comes with a family $\{\omega^t\}_{t\in \RR}=\{\omega^t_z\}_{z\in \ZZ^2,t\in \RR}$ of dynamically evolving i.i.d.\ $\mathrm{Exp}(1)$ variables, we shall take $F\sim \mathrm{Exp}(1)$. Also, in our setting, we have a Poisson clock at each $z\in \ZZ^2$ according to which the weights are being resampled-- that is, at time $t$, there is a probability $1-e^{-t}$ of having already resampled the weight at any given vertex. Thus, in order to respect the above time change, we consider the configuration $\tilde \omega^r=\omega^{-\log(1-r)}$ and note that $\tilde \omega^r$ is obtained by resampling each vertex of $\tilde \omega^0$ independently with probability $r$.

Now, for points $p\leq q\in \ZZ^2$ and $r\in (0,1)$, we shall work with $\widetilde{T}_{p}^{q,r}=T_{p}^{q,-\log(1-r)}$, which we note is measurable with respect to the finitely many values $\widetilde{\omega}_z^r$ for all $z\in \ZZ^2$ satisfying $p\leq z\leq q$. %

In view of the above, we define the function $f_{p,q}$ such that $\widetilde{T}_{p}^{q,r}=f_{p,q}(\widetilde{\omega}^r)$. Thus, we can now consider the quantities $D_z^xf_{p,q}$ from \eqref{eq:657} for all $p\leq z\leq q$. The following result is a minor modification of \cite[Lemma 3.3]{ADS24}.
\begin{proposition}
  \label{prop:11.1}
  There exists a constant $c>0$ such that for all $p_1\leq q_1, p_2\leq q_2$, and all $z\in \ZZ^2$ satisfying $p_1,p_2\leq z\leq  q_1,q_2$, and all $r\in (0,1)$, we have
  \begin{equation}
    \label{eq:56}
    \mathrm{Inf}_z^{f_{p_1,q_1},f_{p_2,q_2}}(r)\geq c \PP(z\in \widetilde{\Gamma}_{p_1}^{q_1,0}\cap \widetilde{\Gamma}_{p_2}^{q_2,r}).
  \end{equation}
\end{proposition}
We note that the source \cite{ADS24} proves the above result for the special case $(p_1,q_1)=(p_2,q_2)$. However, by inspecting the proof, it can be seen that the proof generalises verbatim to yield the above result. The following result is a consequence of the above propositions along with a simple time change; note that for a finite set $A\subseteq \ZZ^2$, $|A|$ simply refers to the cardinality of $A$.
\begin{lemma}
  \label{lem:26}
  There exists a constant $c>0$ such that for all $p_1\leq q_1, p_2\leq q_2$, we have
  \begin{equation}
    \label{eq:369}
    \mathrm{Cov}(T_{p_1}^{q_1},T_{p_2}^{q_2})\geq c\int_0^\infty (\EE|\Gamma_{p_1}^{q_1,0}\cap \Gamma_{p_2}^{q_2,t}|) e^{-t}dt.
  \end{equation}
\end{lemma}
\begin{proof}
  By using Proposition \ref{prop:10.1} followed by Proposition \ref{prop:11.1}, we obtain
  \begin{align}
    \label{eq:57}
    \mathrm{Cov}(T_{p_1}^{q_1},T_{p_2}^{q_2})&= \EE[\widetilde{T}_{p_1}^{q_1,0}\widetilde{T}_{p_2}^{q_2,0}]-\EE[\widetilde{T}_{p_1}^{q_1,0}\widetilde{T}_{p_2}^{q_2,1}]\nonumber\\
    &=\int_0^1 \sum_{z\in \ZZ^2}\mathrm{Inf}_z^{f_{(p_1,q_1)},f_{(p_2,q_2)}}(r)dr\nonumber\\
                                             &=\sum_{z: p_1,q_1\leq z\leq p_2,q_2}\int_0^1\mathrm{Inf}_z^{f_{(p_1,q_1)},f_{(p_2,q_2)}}(r)dr\nonumber\\
    &\geq c\int_0^1\EE(| \widetilde{\Gamma}_{p_1}^{q_1,0}\cap \widetilde{\Gamma}_{p_2}^{q_2,r}|)dr=c \int_0^\infty(\EE|\Gamma_{p_1}^{q_1,0}\cap \Gamma_{p_2}^{q_2,t}|) e^{-t}dt.
  \end{align}
To obtain the last term above, we have performed the substitution $r=1-e^{-t}$.
\end{proof}

\section{Last passage percolation preliminaries}
\label{sec:usef-estim-last}
In this section, we shall collect certain useful estimates and results for exponential LPP.

\subsection{Transversal fluctuation estimates}
\label{sec:transfluc}
We shall frequently need estimates controlling the deviation of geodesics in LPP from the straight line joining their endpoints; recall that for points $p\leq q\in \ZZ^2$ and $r\in [\![\phi(p), \phi(q)]\!]$, we can define $\Gamma_p^q(r)$ as described in Section \ref{sec:model-exp}.
\begin{proposition}[{\cite[Theorem 11.1]{BSS14}, \cite[Proposition C.9]{BGZ19}}]
  \label{prop:39}
  Fix $K>1$. There exist constants $C,c$ such that for any point $q\in \ZZ^2$ satisfying $\phi(q)=2n$ and $\slope(\0,q)\in (K^{-1},K)$, we have
  \begin{equation}
    \label{eq:462}
    \PP(\sup_{r\in [\![0,2n]\!]}|\Gamma_{\0}^q(r)-\psi(rq/2n)|\geq \alpha n^{2/3})\leq Ce^{-c\alpha^3}.
  \end{equation}
\end{proposition}
In fact, we shall also require a stronger mesoscopic transversal fluctuation estimate for exponential LPP, and we now state this.
\begin{proposition}[{\cite[Proposition 2.1]{BBB23}}]
  \label{prop:40}
  Fix $K>1$. Then there exist constants $C,c$ such that for all points $q\in \ZZ^2$ satisfying $\phi(q)=2n$ and $\slope(\0,q)\in (K^{-1},K)$, all $r\in [\![0,n]\!]$ and all $\alpha>0$, we have
  \begin{equation}
    \label{eq:463}
    \PP(|\Gamma_{\0}^{q}(r)-\psi(rq/2n)|\geq \alpha r^{2/3})\leq Ce^{-c\alpha^3}.
  \end{equation}
\end{proposition}

\subsection{Passage time estimates in exponential LPP}
\label{sec:parallelogram}
We shall require a few moderate deviation estimates for passage times in exponential LPP. First, we state state an estimate for point-to-point passage times.
\begin{proposition}[{\cite[Theorem 2]{LR10}}]
  \label{prop:41}
  Fix $K>0$. There exist constants $C_1,c_1,C_2,c_2>0$ such that for all $m,n$ sufficiently large which additionally satisfy $\slope(\0,(m,n))\in (K^{-1},K)$ and all $\alpha>0$, we have
  \begin{enumerate}
  \item $\PP(T_{\0}^{(m,n)}-(\sqrt{m}+\sqrt{n})^2\geq \alpha n^{1/3})\leq C_1e^{-c_1\min\{\alpha^{3/2},\alpha n^{1/3}\}}$.
  \item $\PP(T_{\0}^{(m,n)}-(\sqrt{m}+\sqrt{n})^2\leq -\alpha n^{1/3})\leq C_2e^{-c_2\alpha^3}$.
  \end{enumerate}
\end{proposition}
We note that a finer version of the above with the optimal values of the constants $c_1,c_2$ identified is now available (\cite[Theorem 1.6]{BBBK24})-- however, we do not state it here since the above shall suffice for our application. Before moving on, we note that Proposition \ref{prop:41} implies that
\begin{equation}
  \label{eq:465}
  |\EE T_{\0}^{(m,n)}- (\sqrt{m}+\sqrt{n})^2|\leq Cn^{1/3}
\end{equation}
for all $m,n$ large enough, with $C$ being a positive constant depending only on $K$. 

Later in the paper, we shall also require corresponding moderate deviation estimates for `restricted' passage times which we now define. For a subset $U\subseteq \RR^2$, and points $u\leq v\in \ZZ^2$, we use $T_{u}^{v}\lvert_{U}$ to denote
  \begin{equation}
    \label{eq:403}
    T_{u}^{v}\lvert_{U}=\max_{\gamma: u\rightarrow v, \gamma\setminus\{u,v\}\subseteq U} \wgt(\gamma),
  \end{equation}
  with the convention that $T_{u}^{v}\lvert_{U}=-\infty$ if there does not exist any path $\gamma$ as mentioned above. We note that, as per the above definition and the definition \eqref{eq:334} of the weight of a path, the random variable $T_{u}^{v}\lvert_{U}$ is measurable with respect to the vertex weights $\{\omega_z\}_{z\in U\cap \ZZ^2} \cup\{\omega_u,\omega_v\}$. We now have the following useful result.
  \begin{proposition}[{\cite[Theorem 4.2]{BGZ19}}]
    \label{prop:42}
    Fix $K>1,L>0$. Let $q\in \ZZ^2$ be such that $\phi(q)=2n$, $\slope(0,q)\in (K^{-1},K)$. Consider the point $p_{n,L}$ satisfying $\phi(p_{n,L})=0$ and $\psi(p_{n,L})=Kn^{2/3}$ and let $\ell_{n,L}$ denote the line segment joining $-p_{n,L}$ and $p_{n,L}$. Let $U_{n,L}^q$ denote the parallelogram with $\ell_{n,L}$ and $q+\ell_{n,L}$ as one pair of opposite sides. Then there exist constants $C,c$ such that for all $n,\alpha>0$ and with $z,z'$ ranging over the sets $\ell_{n,L}$ and $q+\ell_{n,L}$ respectively, we have
    \begin{enumerate}
    \item $\PP(\max_{z,z'}(T_z^{z'}-\EE T_z^{z'})\geq \alpha n^{1/3})\leq Ce^{-c\min\{\alpha^{3/2},\alpha n^{1/3}\}}$,
    \item $\PP(\min_{z,z'}(T_z^{z'}-\EE T_z^{z'})\leq -\alpha n^{1/3})\leq Ce^{-c\alpha^3}$,
    \item $\PP(\min_{z,z'}(T_z^{z'}\lvert_{U_{n,L}^q}-\EE T_z^{z'})\leq -\alpha n^{1/3})\leq Ce^{-c\alpha}$.
    \end{enumerate}
  \end{proposition}
\subsection{A local diffusivity estimate in exponential LPP}
\label{sec:diff-est}
The following regularity estimate for the profile of distances from a fixed point will be useful to us.
\begin{proposition}[{\cite[Theorem 3]{BG21}}]
  \label{prop:37}
  Fix $K>1$. Then there exists a constant $C>0$ such that for any points $q,q'$ satisfying $\phi(q)=\phi(q')=2n$, $\slope(\0,q)\in (K^{-1},K)$ and $|\psi(q-q')|\leq Cn^{2/3}$, we have for positive constants $C',c'$,
  \begin{equation}
    \label{eq:417}
    \PP( |T_{\0}^{q}-T_{\0}^{q'}|\geq \alpha|\psi(q-q')|^{1/2})\leq C'e^{-c'\alpha^{4/9}}.
  \end{equation}
\end{proposition}
We note that exponent $4/9$ in the above is not optimal, but it shall suffice for our application.
\subsection{An invariance for the joint distribution of passage times}
\label{sec:distr-symm-joint}
As part of the proof of Theorem \ref{thm:4}, we shall need to estimate the covariances between certain passage times $T_{p_1}^{q_1}, T_{p_2}^{q_2}$ in exponential LPP. Building up on corresponding symmetries for Brownian LPP obtained in \cite{BGW22}, the work \cite{Dau22+} established a certain distributional symmetry in exponential LPP for the joint law of $T_{p_1}^{q_1}, T_{p_2}^{q_2}$ depending on the mutual orientation of the pairs $(p_1,q_1), (p_2,q_2)$, and we now state this.
\begin{proposition}[{\cite[Theorem 1.2]{Dau22+}}]
  \label{prop:50}
  Let $p_1\leq q_1\in \ZZ^2$ and $p_2\leq q_2\in \ZZ^2$. Suppose $\mathbf{c}\in \ZZ^2$ is such that every up-right path $\gamma\colon p_1\rightarrow q_1$ must non-trivially intersect every up-right path $\eta\colon p_2\rightarrow q_2$ and $\eta'\colon p_2+\mathbf{c}\rightarrow q_2+ \mathbf{c}$. Then we have the distributional equality
  \begin{equation}
    \label{eq:582}
    (T_{p_1}^{q_1}, T_{p_2}^{q_2})\stackrel{d}{=} (T_{p_1}^{q_1}, T_{p_2+\mathbf{c}}^{q_2+\mathbf{c}}).
  \end{equation}
\end{proposition}
We note that in the above, we have only stated a special case of the result in \cite{Dau22+}, and the result obtained therein is more general. We shall use Proposition \ref{prop:50} in the proof of Theorem \ref{thm:4} to make certain covariance computations tractable.

\section{Outline of the proof}
\label{sec:outline-proofs}
We now outline the proof of Theorem \ref{thm:4}-- this proof is based on the second moment method. Throughout this section, we shall consider the points $u_j,v_k$ (see Figure \ref{fig:ujvk}) defined by $\phi(u_j)=-n, \psi(u_j)=-\theta n+jn^{2/3}$ and $\phi(v_k)=n, \psi(v_k)=\theta n + kn^{2/3}$. Since $u_j\in \ell_{-n,\varepsilon}^\theta,v_k\in \ell_{n,\varepsilon}^\theta$ for $|j|,|k|\leq \varepsilon n^{1/3}$, to prove Theorem \ref{thm:4}, we need only show that there exists $C>0$ such that for all $n$, we have
\begin{equation}
  \label{eq:521}
  \PP(\exists j\in [\![-\varepsilon n^{1/3},\varepsilon n^{1/3}]\!], t\in [0,1]: \0\in \Gamma_{u_j}^{v_{-j},t})\geq C(\log n)^{-1}.
\end{equation}
\begin{figure}
  \centering
  \includegraphics[width=0.5\linewidth]{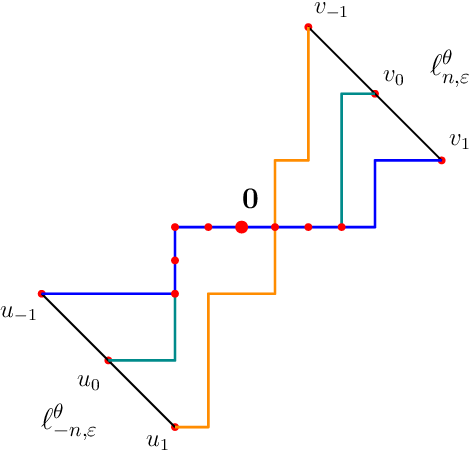}
  \caption{We mark $n^{2/3}$-equispaced points $u_j,v_k$ on the $2\varepsilon n$ length line segments $\ell_{-n,\varepsilon}^\theta,\ell_{n,\varepsilon}^\theta$ respectively. Subsequently, we consider the geodesics $\Gamma_{u_j}^{v_{-j},t}$ for all $t\in [0,1]$ and track whether $\0\in \Gamma_{u_j}^{v_{-j},t}$ occurs. For the case $\theta=0$ and when $\varepsilon n^{1/3}=1$, this figure shows a snapshot at a particular $t\in [0,1]$ at which $\0\in \Gamma_{u_0}^{v_0,t}\cap \Gamma_{u_{-1}}^{v_1,t}$ but $\0\notin \Gamma_{u_1}^{v_{-1},t}$. Thus, this value of $t$, contributes to precisely two of the three integrals appearing in the sum in the definition of $X_n$ (see \eqref{eq:522}). Further, note that in the figure, the overlap set $\Gamma_{u_0}^{v_0,t}\cap \Gamma_{u_{-1}}^{v_1,t}$ consists of exactly $8$ vertices.}
  \label{fig:ujvk}
\end{figure}
Now, we consider the statistic $X_n$ defined by
\begin{equation}
  \label{eq:522}
 X_n= \sum_{|j|\leq \varepsilon n^{1/3}}\int_{0}^{1} \ind(\0 \in \Gamma_{u_j}^{v_{-j},t})dt,
\end{equation}
and our strategy to establish \eqref{eq:521} is to obtain the bounds $\EE X_n\geq C_1n^{-1/3}, \EE X_n^2\leq C_2n^{-2/3}\log n$ for some constants $C_1,C_2$ depending on $\varepsilon,\theta$. Having obtained this, the second moment method would immediately yield that for some constant $C$, 
\begin{equation}
  \label{eq:523}
  \PP(X_n>0)\geq \frac{(\EE X_n)^2}{\EE X_n^2}\geq C(\log n)^{-1},
\end{equation}
and this would complete the proof. Thus the primary goal now is to estimate the first and second moment of $X_n$.

Now, estimating $\EE X_n$ is easy-- indeed, by Fubini's theorem and the stationarity of the dynamics, we have
\begin{equation}
  \label{eq:524}
  \EE X_n= \sum_{|j|\leq \varepsilon n^{1/3}} \PP(\0 \in \Gamma_{u_j}^{v_{-j}}).
\end{equation}
Owing to the $n^{2/3}$ transversal fluctuation scale of geodesics, each term on the right hand side must be (see \cite[Theorem 2]{BB24}) at least $Cn^{-2/3}$ for some constant $C$. Since the sum is over $2\varepsilon n^{1/3}$ many terms, this yields $\EE X_n \geq 2C\varepsilon n^{-1/3}=C_1 n^{-1/3}$, where $C_1$ is a constant depending on $\theta,\varepsilon$.
\subsection{The second moment of $X_n$ in terms of covariances}
\label{sec:out-mom-cov}
The goal now is to obtain an $O(n^{-2/3} \log n)$ upper bound on $\EE X_n^2$, and this is much more involved. First, by basic algebra and by using the stationarity of the dynamics, one obtains
\begin{equation}
  \label{eq:525}
  \EE X_n^2\leq 2\sum_{|j_1|,|j_2|\leq \varepsilon n^{1/3}}\int_{0}^{1} \PP(\0 \in \Gamma_{u_{j_1}}^{v_{-j_1},0}\cap \Gamma_{u_{j_2}}^{v_{-j_2},t})dt.
\end{equation}
Further, we write
\begin{align}
  \label{eq:528}
  \PP(\0 \in \Gamma_{u_{j_1}}^{v_{-j_1},0}\cap \Gamma_{u_{j_2}}^{v_{-j_2},t}) &= \PP(\0 \in \Gamma_{u_{j_1}}^{v_{-j_1},0})\PP(\0\in \Gamma_{u_{j_2}}^{v_{-j_2},t}\vert \0 \in \Gamma_{u_{j_1}}^{v_{-j_1},0})\nonumber\\
  &\leq C_1 n^{-2/3} \PP(\0\in \Gamma_{u_{j_2}}^{v_{-j_2},t}\vert \0 \in \Gamma_{u_{j_1}}^{v_{-j_1},0}),
\end{align}
where the $C_1n^{-2/3}$ bound for the term $\PP(\0 \in \Gamma_{u_{j_1}}^{v_{-j_1},0})$ follows from the discussion earlier for the first moment. However, the term $\PP(\0\in \Gamma_{u_{j_2}}^{v_{-j_2},t}\vert \0 \in \Gamma_{u_{j_1}}^{v_{-j_1},0})$ is harder to analyse. As we shall now describe in a very rough heuristic argument, there is a way to connect this quantity to overlaps of geodesics. First, by using a transversal fluctuation argument, it can be shown that for any $t$, if we define
\begin{equation}
  \label{eq:526}
  R(t,j_1,j_2)= \min\{r: \Gamma_{u_{j_1}}^{v_{-j_1},0}\cap \Gamma_{u_{j_2}}^{v_{-j_2},t}\subseteq \{z:|\phi(z)|\leq r\}\},
\end{equation}
then the natural scale of $R(t,j_1,j_2)$ is $n/(|j_1-j_2|)$ in the sense that we have stretched exponential tail estimates in $\alpha$ for the quantity $\PP(R(t,j_1,j_2)\geq \alpha (n/|j_1-j_2|))$ (see Figure \ref{fig:overlap}). Indeed, we know that the geodesics $\Gamma_{u_{j_1}}^{v_{-j_1},0}$ and $\Gamma_{u_{j_2}}^{v_{-j_2},0}$ lie in a $O(n^{2/3})$ spatial window of their respective lines $\LL_{u_{j_1}}^{v_{-j_1}}$ and $\LL_{u_{j_2}}^{v_{-j_2}}$. Locally using the notation $\LL_{u_{j_1}}^{v_{-j_1}}(r)$ to denote the unique point on $\LL_{u_{j_1}}^{v_{-j_1}}\cap \{z: \phi(z)=r\}$, it can be checked that $|\LL_{u_{j_1}}^{v_{-j_1}}(r)-\LL_{u_{j_2}}^{v_{-j_2}}(r)|=O(n^{2/3})$ is equivalent to $|r|=O(n/|j_1-j_2|)$. This justifies the presence of the term $n/(|j_1-j_2|)$ in the discussion above-- thus, to summarise, the intersection $\Gamma_{u_{j_1}}^{v_{-j_1},0}\cap \Gamma_{u_{j_2}}^{v_{-j_2},t}$ lies in an $O(n/(|j_1-j_2|+1))$ length window of $\{z: \phi(z)=0\}$, where we replaced $|j_1-j_2|$ by $|j_1-j_2|+1$ in order to simultaneously handle the case when $j_1=j_2$.

 Now, in view of the above, it is tempting to assume that
\begin{equation}
  \label{eq:529}
  \PP(\0\in \Gamma_{u_{j_2}}^{v_{-j_2},t}\vert \0 \in \Gamma_{u_{j_1}}^{v_{-j_1},0}) \approx \frac{\EE |\Gamma_{u_{j_2}}^{v_{-j_2},t}\cap \Gamma_{u_{j_1}}^{v_{-j_1},0}|}{n/(|j_1-j_2|+1)},
\end{equation}
In particular, we are assuming in the above that for all $|s|\leq n/(|j_1-j_2|+1)$, the probabilities $\PP(\Gamma_{u_{j_2}}^{v_{-j_2},t}(s)=  \Gamma_{u_{j_1}}^{v_{-j_1},0}(s))$ are comparable, where we are interpreting geodesics as functions as per the notation defined in Section \ref{sec:model-exp}. Assuming the above heuristic connection to geodesic overlaps and using \eqref{eq:528} and \eqref{eq:525}, we obtain that for some constants $C_1,C_2$, we have
\begin{align}
  \label{eq:530}
  \EE X_n^2&\leq C_1n^{-2/3}\sum_{|j_1|,|j_2|\leq \varepsilon n^{1/3}} \frac{\int_0^1\EE |\Gamma_{u_{j_2}}^{v_{-j_2},t}\cap \Gamma_{u_{j_1}}^{v_{-j_1},0}|dt}{n/(|j_1-j_2|+1)}\nonumber\\
  &\leq C_1n^{-2/3}\sum_{|j_1|,|j_2|\leq \varepsilon n^{1/3}} \frac{C_2\Cov(T_{u_{j_1}}^{v_{-j_1}}, T_{u_{j_2}}^{v_{-j_2}})}{n/(|j_1-j_2|+1)},
\end{align}
where to obtain the last line above, we have used the dynamical Russo-Margulis formula (Lemma \ref{lem:26}), which connects integrals of geodesic overlaps over time to covariances in the static model. We emphasize that the final expression above is just in terms of static exponential LPP.

\subsection{Estimating the covariances $\Cov(T_{u_{j_1}}^{v_{-j_1}}, T_{u_{j_2}}^{v_{-j_2}})$}
\label{sec:out-cov}
The goal now is to show that
\begin{equation}
  \label{eq:531}
  |\Cov(T_{u_{j_1}}^{v_{-j_1}}, T_{u_{j_2}}^{v_{-j_2}})|\leq  \frac{Cn^{2/3}}{(|j_1-j_2|+1)^2}.
\end{equation}
This would be sufficient since then by \eqref{eq:530}, we would obtain
\begin{equation}
  \label{eq:532}
  \EE X_n^2\leq C'n^{-2/3}\sum_{|j_1|,|j_2|\leq \varepsilon n^{1/3}}\frac{n^{-1/3}}{|j_1-j_2|+1}=O(n^{-2/3}\log n).
\end{equation}
\begin{figure}
  \centering
  \includegraphics[width=0.65\linewidth]{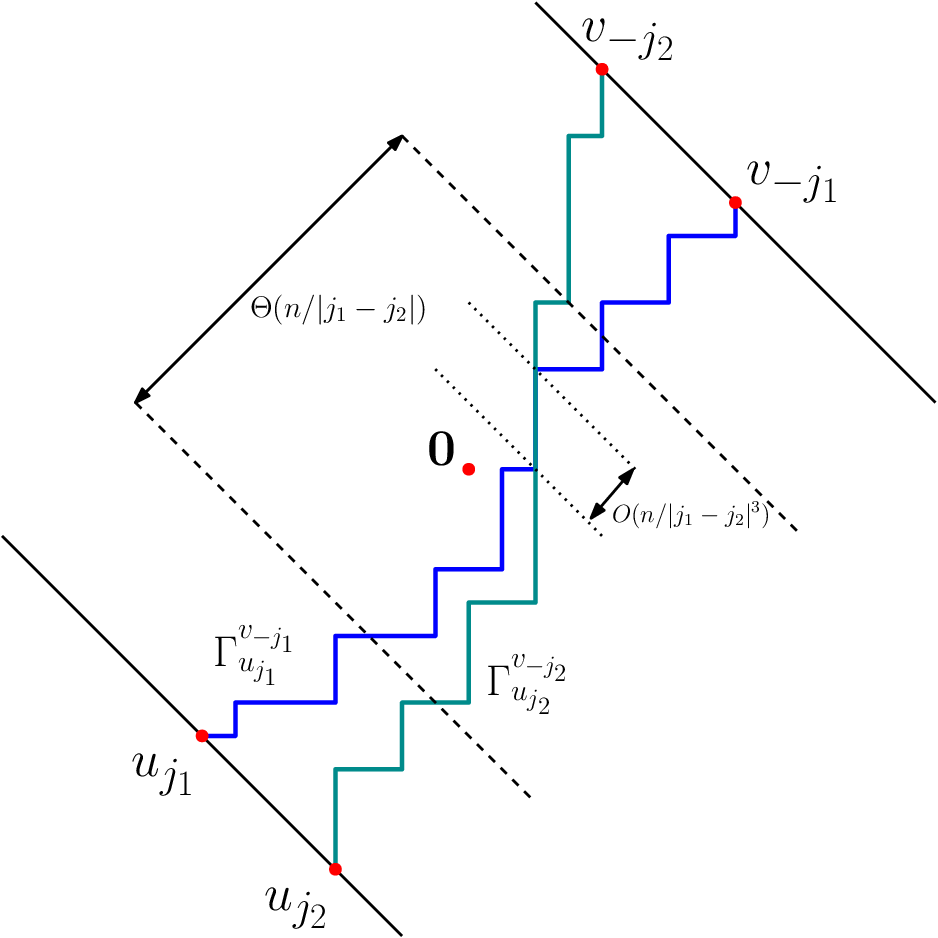}
  \caption{For static exponential LPP and for large $|j_1-j_2|$, by the results from \cite{BBB23}, we expect the geodesics $\Gamma_{u_{j_1}}^{v_{-j_1}}, \Gamma_{u_{j_2}}^{v_{-j_2}}$ to overlap for an $O(n/|j_1-j_2|^{3})$ contiguous stretch located at a random location in a larger region of length $\Theta(n/|j_1-j_2|)$ about $\{z: \phi(z)=0\}$. Thus, by a KPZ scaling heuristic, we expect $\Cov(T_{u_{j_1}}^{v_{-j_1}}, T_{u_{j_2}}^{v_{-j_2}})$ to originate entirely from the overlap, and thus be $O((n/|j_1-j_2|^3)^{2/3})= O(n^{2/3}/|j_1-j_2|^2)$. In contrast, for dynamical LPP, the overlap set $|\Gamma_{u_{j_1}}^{v_{-j_1},0}\cap \Gamma_{u_{j_2}}^{v_{-j_2},t}|$ is no longer necessarily contiguous if $t\neq 0$. However, we still expect that $\Gamma_{u_{j_1}}^{v_{-j_1},0}\cap \Gamma_{u_{j_2}}^{v_{-j_2},t}$ is located within a $\Theta(n/|j_1-j_2|)$ stretch around $\{z: \phi(z)=0\}$, that is, between the dashed lines above.}
  \label{fig:overlap}
\end{figure}
Thus, we need only establish \eqref{eq:531}. We now, very heuristically, discuss why one could expect \eqref{eq:531} to be true. In the work \cite{BBB23} via a transversal fluctuation argument, it is established that the expected overlap $\EE|\Gamma_{u_{j_1}}^{v_{-j_1}}\cap \Gamma_{u_{j_2}}^{v_{-j_2}}|$ is $O(n/|j_1-j_2|^{3-o(1)})$ (see Figure \ref{fig:overlap}); one might expect that the $o(1)$ correction above is an artifact of the proof and that the right order of the overlap ought to be $O(n/|j_1-j_2|^{3})$. For instance, for $|j_1-j_2|=O(1)$, the geodesics are expected to have linear overlap, while for $|j_1-j_2|\sim n^{1/3}$, the geodesics go in macroscopically different directions and are expected to have constant overlap. 

Thus, heuristically, one could expect the entire covariance $|\Cov(T_{u_{j_1}}^{v_{-j_1}}, T_{u_{j_2}}^{v_{-j_2}})|$ as originating from the weight of the overlapping region $\Gamma_{u_{j_1}}^{v_{-j_1}}\cap \Gamma_{u_{j_2}}^{v_{-j_2}}$. That is, intuitively, except for an $O(n/|j_1-j_2|^3)$ length stretch, the geodesics $\Gamma_{u_{j_1}}^{v_{-j_1}}$, $\Gamma_{u_{j_2}}^{v_{-j_2}}$ go via disjoint regions of the space and thus the contributions to the passage times $T_{u_{j_1}}^{v_{-j_1}},T_{u_{j_2}}^{v_{-j_2}}$ from these regions are ``independent'' and do not contribute to the covariance. Thus, one might expect that
\begin{equation}
  \label{eq:533}
  \Cov(T_{u_{j_1}}^{v_{-j_1}},T_{u_{j_2}}^{v_{-j_2}})\sim \Var(\wgt(\Gamma_{u_{j_1}}^{v_{-j_1}}\cap \Gamma_{u_{j_2}}^{v_{-j_2}})).
\end{equation}
Finally, since $|\Gamma_{u_{j_1}}^{v_{-j_1}}\cap \Gamma_{u_{j_2}}^{v_{-j_2}}|$ is of the order $n/|j_1-j_2|^3$ as discussed above, the right hand side of \eqref{eq:533} should be of the order $(n/|j_1-j_2|^3)^{2/3}=n^{2/3}/|j_1-j_2|^2$ via the KPZ 1:2:3 scaling (see Figure \ref{fig:overlap}). This justifies \eqref{eq:531}.

We note that the quadratic covariance decay exponent exponent appearing in \eqref{eq:531} also appears in the continuum theory in the covariance decay estimates for the $\textrm{Airy}_2$ process (\cite[Theorem 1.6]{AM05}, \cite{Wid04}). Indeed, \eqref{eq:532} can be considered to be a discrete version of the above.
\subsection{Remarks on the differences in the actual proof}
\label{sec:outrem}
Since there are some difficulties in making \eqref{eq:529} precise, the actual argument is more complicated than the outline presented above. The approach we take is to use an averaging argument-- this requires using a more complicated definition of $X_n$, and in the actual proof, we let $\fp_n$ be a uniformly random point (independent of the LPP) in an $c_\varepsilon n$ sized square around $\0$ and then define
\begin{equation}
  \label{eq:534}
  X_{n}=\sum_{|j|,|k|\leq \varepsilon n^{1/3}}\int_{0}^{1} \ind(\fp_n \in \Gamma_{u_j}^{v_k,t})dt.
\end{equation}
While the presence of two variables $j,k$ instead of the single variable $j$ present earlier makes the computations more complicated, with the new definition, the step \eqref{eq:529} in the computation of $\EE X_n^2$ can be bypassed. Indeed, this is because we have
\begin{equation}
  \label{eq:535}
  \PP(\fp_n\in \Gamma_{u_{j_1}}^{v_{k_1}}\cap \Gamma_{u_{j_2}}^{v_{k_2}})\leq \frac{\EE |\Gamma_{u_{j_1}}^{v_{k_1}}\cap \Gamma_{u_{j_2}}^{v_{k_2}}|}{(c_\varepsilon n)^2},
\end{equation}
and thus the above brings the expected overlaps into play without requiring an analogue of \eqref{eq:529}. However, we caution that due to the above new definition of $X_n$, there is an extra de-randomization argument required at the end to obtain \eqref{eq:521} which has the point $\0$ instead of $\fp_n$. This is done in Section \ref{sec:obta-exist} by rerooting at the point $\fp_n$ and exploiting the translation invariance of LPP along with some transversal fluctuation estimates. Before moving on, we note that averaging arguments such as the one described above have often been used in the LPP literature, with some examples being \cite{BSS14, BHS22, BB24, BBB23}.

\subsection{The heuristics behind Conjecture \ref{conj:1}}
\label{sec:are-upper-bounds}
We now discuss why we expect that $\dim \scT=0$ almost surely. For $\delta>0$ and $n\in \ZZ$, define $\ell_{n,\delta}=\{p\in \ZZ^2:\phi(p)=n, |\psi(p)|\leq (1-\delta) |n|\}$.
We assume that non-trivial bigeodesics, if they exist, always stay macroscopically away from the coordinate axis in the sense that almost surely, for any $t$ admitting a non-trivial bigeodesic $\Gamma^t$, we have $-1<\liminf_{m\rightarrow \infty}\Gamma^t(m)/m\leq \limsup_{m\rightarrow \infty}\Gamma^t(m)/m<1$. Note that the above is non-trivial-- it was shown for static exponential LPP in \cite[Section 5]{BHS22} and is established for dynamical Brownian LPP in the companion paper \cite{Bha25+} (see Proposition 13 therein).

In view of the above, by a first moment argument, in order to show that $\dim \scT=0$, it suffices to establish that for any fixed $\delta\in (0,1)$, we have
\begin{equation}
  \label{eq:8}
  \PP(\0\in \bigcup_{p\in \ell_{-n,\delta},q\in \ell_{n,\delta},t\in [0,n^{-1/3}]}\Gamma_{p}^{q,t})=O(n^{-1/3+o(1)}).
\end{equation}
Now, let $\cR_n =\{z\in \RR^2: \phi(z)\in [-n/2,n/2]\}$. By an argument resembling the one outlined in Section \ref{sec:outrem} where one randomizes the deterministic point $\0$ to a random point in a $\Theta(n)$ sized box around $\0$ contained wholly in $\cR_n$, it suffices to establish that
\begin{equation}
  \label{eq:7}
  \EE|\cR_n\cap \bigcup_{p\in \ell_{-n,\delta},q\in \ell_{n,\delta},t\in [0,n^{-1/3}]}\Gamma_{p}^{q,t}|=O(n^{5/3+o(1)}).
\end{equation}
To show the above, it suffices to establish the following. Consider two line segments $L_1\subseteq \ell_{-n,\delta}, L_2\subseteq \ell_{n,\delta}$ of length $n^{2/3}$ each. Then we have
\begin{equation}
  \label{eq:2}
  \EE|\cR_n\cap \bigcup_{p\in L_1,q\in L_2,t\in [0,n^{-1/3}]}\Gamma_{p}^{q,t}|=O(n^{1+o(1)}).
\end{equation}
Indeed, the above would imply \eqref{eq:7} as we can cover $\ell_{-n,\delta}\times \ell_{n,\delta}$ by $n^{2/3}$ such pairs $L_1\times L_2$. Note that instead of \eqref{eq:2}, one can analogously consider the static LPP quantity $\EE|\cR_n\cap \bigcup_{p\in L_1,q\in L_2}\Gamma_{p}^{q}|$, and this is in fact $\Theta(n)$ as was shown in the coalescence estimates in \cite[Theorem 3.10]{BHS22}. Indeed, intuitively, since the region $\cR_n$ is at a linear distance from the segments $L_1,L_2$, all the geodesics $\Gamma_{p}^{q}$ have undergone coalescence by the time that they reach $\cR_n$.

Thus, the required estimate \eqref{eq:2} can be interpreted as a version of geodesic coalescence that holds for dynamical LPP right up to the subcritical time scale $n^{-1/3}$ and this is plausible as we intuitively (see Proposition \ref{prop:44}) expect geodesics between points in $L_1,L_2$ to remain stable (at the level of overlaps) until the threshold $n^{-1/3-o(1)}$.

The above provides the intuitive justification behind Conjecture \ref{conj:1}. Thus, it seems delicate to, even heuristically, figure out whether the set $\scT$ should a.s.\ be empty or non-empty, since the above heuristic calculation yields a first moment estimate corresponding to Hausdorff dimension $0$ while Theorem \ref{thm:4} yields a sub-polynomial lower bound on the probability of there existing long geodesics passing through $\0$ at some time in $[0,1]$.

\section{Open questions}
\label{sec:open-questions}
In this short section, we collect a few open questions that are raised by this work.
\begin{question}
  \label{que:2}
  Can the $\Omega(1/\log n)$ lower bound in Theorem \ref{thm:4} be upgraded to an $\Omega(1)$ lower bound? 
\end{question}
If the above holds, then it is plausible that with an additional ergodicity argument, one would obtain the a.s.\ existence of exceptional times. A possible approach to tackle Question \ref{que:2} is to attempt a more refined version of the second moment argument used to prove Theorem \ref{thm:4}. Indeed, as is evident in the second moment argument proof of Theorem \ref{thm:4}, the occurrence of $\0\in \Gamma_{u_j}^{v_{-j},t}$ for some one $j$ and $t\in [0,1]$ makes it more likely that there are other $j',t'$ for which we also have $\0\in \Gamma_{u_{j'}}^{v_{-j'},t'}$. In other words, the presence of one long geodesic passing through $0$ makes it more likely that there are more such geodesics passing via $0$, albeit in possibly different directions and at different times. It is plausible that by carefully employing a second moment argument with a carefully chosen weighted statistic $\widetilde X_n$, one might be able to offset the above effect and have $\EE\widetilde X_n^2\approx (\EE \widetilde X_n)^2$. As of now, we have not been able to make the above strategy work, but we hope to pursue it further in the future. We remark that in the literature, weighted second moment arguments have often been successfully used for tackling questions related to random constraint satisfaction problems (see e.g.\ \cite{AP04,CP13,DSS22}).

We now have the following natural question.
\begin{question}
Prove that Conjecture \ref{conj:1} holds.
\end{question}
As discussed in Section \ref{sec:are-upper-bounds}, in order to prove the above for exponential LPP, it would suffice to prove \eqref{eq:7}. It is possible to envisage a strategy wherein one relates the event $|\cR_n\cap \bigcup_{p\in L_1,q\in L_2,t\in [0,n^{-1/3-\delta}]}\Gamma_{p}^{q,t}|\geq n^{1+\varepsilon}$ to a ``multi-peak'' event in the static LPP $T^0$, in a manner similar to \cite{GH24}, where such a strategy was used for Brownian LPP in a comparatively simpler setting where geodesics at only two times $t=0$ and $t=n^{-1/3-\delta}$ are considered, instead of letting $t$ vary freely over $[0,n^{-1/3-\delta}]$. However, in contrast to the setting of \cite{GH24}, it appears that using the above strategy in our setting would require fine control on how much passage times between two fixed points can change as the dynamics proceeds. Indeed, it would be useful to answer the following question.
\begin{question}
\label{ques:3}
Prove that for any $t\leq n^{-1/3}$, the quantity $\PP(|T_{\0}^{\n,t}-T_{\0}^{\n,0}|\geq \alpha\sqrt{nt})$ decays superpolynomially in $\alpha$.
\end{question}
For Brownian LPP with the noise evolving via the Ornstein-Uhlenbeck dynamics, an application (see \cite[(16)]{GH24}) of the dynamical Russo-Margulis formula shows that the passage times (we locally use tildes to denote Brownian LPP quantities) $\EE|\widetilde{T}_{\0}^{\n,t}-\widetilde{T}_{\0}^{\n,0}|^2=O(nt)$, thereby implying $\PP(|\widetilde{T}_{\0}^{\n,t}-\widetilde{T}_{\0}^{\n,0}|\geq \alpha \sqrt{nt})=O(\alpha^{-2})$; this estimate is used frequently in \cite{GH24}. In order to upgrade this to a superpolynomial estimate, one strategy would be to attempt a finer analysis wherein $|\widetilde{T}_{\0}^{\n,t}-\widetilde{T}_{\0}^{\n,0}|$ is written as a sum of a large number of summands which exhibit sufficient independence.

We now discuss another exciting but speculative question. Recall that in the setting of dynamical critical percolation on the triangular lattice, \cite{GPS10} established that the set of exceptional times when an infinite cluster exists containing the origin a.s.\ has Hausdorff dimension $31/36$. Further, in the work \cite{HPS15}, it was shown that this set of exceptional times naturally comes equipped with a non-trivial local time measure and that at a time sampled from this measure, the percolation configuration agrees with Kesten's incipient infinite cluster \cite{Kes86}. Now, in the context of LPP, recently, there has been significant activity in understanding the environment around a geodesic and this is intimately connected to conditioning on the singular event that the origin lies on a bigeodesic. Indeed, the work \cite{MSZ21} defines a family of measures $\nu^\rho$ for $\rho\in (0,1)$ which can intuitively be thought of as the distribution of exponential LPP conditional on the singular event of there existing a $\rho^{-2}(1-\rho)^{2}$-directed bigeodesic passing through $\0$. In view of the above, one might ask the following speculative question.

\begin{question}
  \label{que:3}
  For dynamical exponential LPP, suppose that there a.s.\ exist exceptional times at which there exists a bigeodesic passing through $\0$. Does this set have a natural local time measure on it? If we let $\mathfrak{t}$ be ``sampled'' from this measure, then can the environment $\{\omega^{\mathfrak{t}}_z\}_{z\in \ZZ^2}$ be written as an explicit convex combination of the measures $\nu^\rho$ from \cite{MSZ21}?
\end{question}

We note that in the scaling limit-- the directed landscape, the work \cite{DSV22} defines a law which should correspond to a directed landscape ``conditioned'' to have a bigeodesic passing through the origin, and thus a version of Question \ref{que:3} can be formulated in this setting as well. However, a natural dynamics on the directed landscape has not been defined yet, and thus the finer question of the investigation of exceptional times for these dynamics is currently out of reach.

Finally, we state a question for an entirely different setting-- that of random planar maps and the associated $\gamma$-Liouville quantum gravity metrics ($\gamma$-LQG). Expected to arise as the scaling limit of discrete random planar map models coupled with a critical statistical physics model, $\gamma$-LQG \cite{She22,DDG21} is a family of continuum planar models of random geometry parametrised by $\gamma\in (0,2)$. Particularly well studied, is the case of \emph{uniform} planar maps \cite{LeGal19}, which corresponds \cite{MS20} to $\gamma=\sqrt{8/3}$ and has a rich integrable theory describing the distances and geodesics arising therein. These models offer a theory of planar random geometry parallel to FPP and LPP and in fact, there are many striking similarities in the behaviour-- for example, all these models exhibit the phenomenon of geodesic coalescence (e.g.\ \cite{GM20}). In fact, the question of bigeodesics has also been investigated in this setting-- it is known \cite[Lemma 4.5]{GPS20} that there a.s.\ do not exist any bigeodesics in $\gamma$-LQG; on the discrete side, for the case of the Uniform Infinite Planar Quadrangulation (UIPQ) \cite{Kri05}, the non-existence of bigeodesics is a direct consequence of the results of \cite{CMM13}. In view of this, one might pose the following question.
\begin{question}
  For $\gamma$-LQG equipped with a ``natural dynamics'', are there any exceptional times at which bigeodesics exist? Similarly, for a dynamical version of a discrete model, say dynamical UIPQ or a dynamical version of the Uniform Infinite Planar Triangulation (UIPT) \cite{AS03}, can there exist exceptional times at which bigeodesics exist?
\end{question}
We note that in the setting of $\gamma$-LQG, the above question seems hard to answer since a natural dynamics on $\gamma$-LQG has not been defined yet. However, the discrete version of the question might be more feasible; for instance, one could work with the edge flip dynamics on the UIPT (see e.g.\ \cite{Bud17}) and attempt to use the integrable structure only present for uniform planar maps as opposed to the general case of $\gamma\neq \sqrt{8/3}$. We note that just as in LPP, there have recently been works investigating the environment around the geodesic in the random planar map setting \cite{Die16,BBG21,Mou24}. In particular, \cite{Die16} constructs the local limit of the UIPQ rooted along a semi-infinite geodesic and \cite{Mou24} constructs the corresponding object for the Brownian plane (or equivalently, the $\sqrt{8/3}$-LQG plane). Thus, it is also possible to pose a version of the highly speculative Question \ref{que:3} in the setting of such planar map models.

\section{Proof of Theorem \ref{thm:4}}
\label{sec:lower-bound}
In this section, we provide the proof of Theorem \ref{thm:4}; note that for the proof, it is sufficient to work with $\varepsilon$ small enough depending on $\theta$. Throughout this section, we shall assume that $\varepsilon$ is small enough to satisfy
\begin{equation}
  \label{eq:724}
  -1< \theta-6\varepsilon< \theta+ 6\varepsilon <1.
\end{equation}
Now, as indicated in Section \ref{sec:outline-proofs}, for $n\in \NN$ and  $j,k\in \RR$, we define the points $u_j,v_k\in \ZZ^2$ by the conditions
\begin{align}
  \label{eq:651}
  &\phi(u_j)=-n, \psi(u_j)\in(-\theta n+jn^{2/3}-2,-\theta n+jn^{2/3}], \nonumber\\
  &\phi(v_k)=n, \psi(v_k)\in (\theta n + kn^{2/3}-2, \theta n + kn^{2/3}].
\end{align}
Note that the above conditions uniquely define the points $u_j,v_k$. For the most part, we shall work with $j,k\in \ZZ$, but on some occasions, we shall work with $j,k\in \RR$ as well. Further, note that, due to the assumption \eqref{eq:724}, we can safely assume throughout that for all $j,k\in [\![-6\varepsilon n^{1/3}, 6\varepsilon n^{1/3}]\!]$ and all $n$, $\slope(u_j,v_k)$ is uniformly bounded away from $0$ and $\infty$-- this will be convenient as we will often use basic estimates (see Section \ref{sec:usef-estim-last}) for exponential LPP.
Define
\begin{equation}
  \label{eq:376}
  \boxx_n= \{p\in \ZZ^2: |\phi(p)|, |\psi(p)|/2 \leq \varepsilon n \}.
\end{equation}
Let $\fp_n$ be a point chosen uniformly in $\boxx_{n}$ independent of the dynamical LPP. Consider the random variable $X_{n}$ defined by
\begin{equation}
  \label{eq:1}
  X_{n}=\sum_{|j|,|k|\leq \varepsilon n^{1/3}}\int_{0}^{1} \ind(\fp_n \in \Gamma_{u_j}^{v_k,t})dt.
\end{equation}
As discussed in Section \ref{sec:outline-proofs}, we shall estimate the first and second moments of $X_n$.

\subsection{The lower bound on the first moment of $X_n$}
\label{sec:lower-bound-first}
\begin{proposition}
  \label{prop:2}
  There exists a constant $C>0$ such that for all $n$, we have $\EE X_n\geq Cn^{-1/3}$.
\end{proposition}
\begin{proof}
Let $T$ be a static exponential LPP which is independent of the point $\fp_n$. By using \eqref{eq:1} along with the linearity of expectation, stationarity of the dynamics and Fubini's theorem, we have
  \begin{equation}
    \label{eq:377}
    \EE X_n=\sum_{|j|,|k|\leq \varepsilon n^{1/3}} \PP(\fp_n\in \Gamma_{u_j}^{v_k}).
  \end{equation}
  As a result, it suffices to show that for some constant $C_1$ and all $j,k$ as above, we have
  \begin{equation}
    \label{eq:378}
       \PP(\fp_n\in  \Gamma_{u_j}^{v_k})\geq C_1 n^{-1}.
     \end{equation}
Note that $|\boxx_n|\leq C_2n^2$ for some positive constant $C_2$. Using this along with the fact that $\fp$ is uniformly sampled from $\boxx_n$ independent of the LPP, we have
     \begin{equation}
       \label{eq:379}
       \PP(\fp_n\in  \Gamma_{u_j}^{v_k})= \frac{1}{|\boxx_{n}|}\EE|\boxx_n\cap \Gamma_{u_j}^{v_k}|\geq  C_2^{-1}n^{-2} \EE|\boxx_n\cap \Gamma_{u_j}^{v_k}|.
     \end{equation}
     As a result of the above, we just need to show that for some $C$ and all $j,k$ as above, we have
     \begin{equation}
       \label{eq:380}
       \EE|\boxx_n\cap \Gamma_{u_j}^{v_k}|\geq Cn.
     \end{equation}
     However, the above is easy to see by using transversal fluctuation estimates. Indeed, by using transversal fluctuation estimates (Proposition \ref{prop:39}) for geodesics, for some constants $C',c'$, on an event $E_n$ with probability at least $1-C'e^{-c'n}$, we have $\Gamma_{u_j}^{v_k}\cap \{p:|\phi(p)|\leq \varepsilon n\}= \Gamma_{u_j}^{v_k}\cap \boxx_n$ for all $|j|,|k|\leq \varepsilon n^{1/3}$. As a result, on the event $E_n$, it is easy to see that we must have $|\Gamma_{u_j}^{v_k}\cap \boxx_n|\geq \lfloor \varepsilon n \rfloor - \lceil -\varepsilon n \rceil\geq \varepsilon n$ for all $n$ large enough. Thus, we have
     \begin{equation}
       \label{eq:381}
       \EE|\boxx_n\cap \Gamma_{u_j}^{v_k}|\geq \EE[|\boxx_n\cap \Gamma_{u_j}^{v_k}|;E_n]\geq \varepsilon n ( 1-C'e^{-c'n})\geq \varepsilon n/2,
     \end{equation}
for all $n$ large enough, and this completes the proof.
\end{proof}

\subsection{The upper bound on the second moment of $X_n$}
\label{sec:upper-bound-second}
As we shall see, controlling the second moment is much harder. Our goal now is to prove the following estimate. %
\begin{proposition}
  \label{prop:1}
  There exists a constant $C$ such that for all $n$, we have $\EE X_n^2\leq C n^{-2/3}\log n$.
\end{proposition}
First, in the following lemma, we use the stationarity of the dynamics to bound the second moment by a relatively tractable expression.
\begin{lemma}
  \label{lem:7}
  We have $\EE X_n^2\leq 2\sum_{|j_1|,|j_2|,|k_1|,|k_2|\leq \varepsilon n^{1/3}}\int_0^1 \PP(\fp_n \in \Gamma_{u_{j_1}}^{v_{k_1},0}\cap \Gamma_{u_{j_2}}^{v_{k_2},t})dt$.
\end{lemma}
\begin{proof}
  By definition, we have
  \begin{align}
    \label{eq:382}
    \EE X_n^2&= \sum_{|j_1|,|j_2|,|k_1|,|k_2|\leq \varepsilon n^{1/3}}\int_0^1\int_0^1 \PP(\fp_n \in \Gamma_{u_{j_1}}^{v_{k_1},s}\cap \Gamma_{u_{j_2}}^{v_{k_2},t})dsdt\nonumber\\
    &= 2\sum_{|j_1|,|j_2|,|k_1|,|k_2|\leq \varepsilon n^{1/3}}\int_{(s,t)\in [0,1]^2,s<t} \PP(\fp_n \in \Gamma_{u_{j_1}}^{v_{k_1},s}\cap \Gamma_{u_{j_2}}^{v_{k_2},t})dsdt.
  \end{align}
  By stationarity, we know that for any fixed $s<t$, we have
  \begin{equation}
    \label{eq:383}
    \PP(\fp_n \in \Gamma_{u_{j_1}}^{v_{k_1},s}\cap \Gamma_{u_{j_2}}^{v_{k_2},t})= \PP(\fp_n \in \Gamma_{u_{j_1}}^{v_{k_1},0}\cap \Gamma_{u_{j_2}}^{v_{k_2},t-s}).
  \end{equation}
  As a result of this and \eqref{eq:382}, we obtain
  \begin{align}
    \label{eq:384}
    \EE X_n^2&= 2\sum_{|j_1|,|j_2|,|k_1|,|k_2|\leq \varepsilon n^{1/3}}\int_{(s,t)\in [0,1]^2,s<t} \PP(\fp_n \in \Gamma_{u_{j_1}}^{v_{k_1},0}\cap \Gamma_{u_{j_2}}^{v_{k_2},t-s}) dsdt\nonumber\\
    &\leq 2\sum_{|j_1|,|j_2|,|k_1|,|k_2|\leq \varepsilon n^{1/3}}\int_0^1\int_0^1 \PP(\fp_n \in \Gamma_{u_{j_1}}^{v_{k_1},0}\cap \Gamma_{u_{j_2}}^{v_{k_2},t})dsdt,
  \end{align}
  and the required inequality now immediately follows from the above.
\end{proof}

We now define the overlap $\cO_{j_1,j_2}^{k_1,k_2}(t,n)=|\Gamma_{u_{j_1}}^{v_{k_1},0}\cap
\Gamma_{u_{j_2}}^{v_{k_2},t}|$, where recall that for a finite set $A\subseteq \RR^2$, $|A|$ simply refers to the cardinality of $A$. %
As an immediate consequence of the dynamical Russo-Margulis formula (Lemma \ref{lem:26}), we have the following result.
\begin{lemma}
  \label{lem:9}
  There exists a constant $C$ such that for all $n$, and all $j_1,j_2,k_1,k_2\in[\![-\varepsilon n^{1/3}, \varepsilon n^{1/3}]\!]$, 
  \begin{equation}
    \label{eq:456}
    \int_0^1 \EE[\cO_{j_1,j_2}^{k_1,k_2}(t,n)]dt\leq C\Cov ( T_{u_{j_1}}^{v_{k_1}},T_{u_{j_2}}^{v_{k_2}}).
  \end{equation}
\end{lemma}
In fact, with the help of the above result, the terms appearing in Lemma \ref{lem:7} can be bounded in terms of covariances.

\begin{lemma}
  \label{lem:8}
  There exists a constant $C$ such that for all $n$ and all $j_1,j_2,k_1,k_2\in [\![-\varepsilon n^{1/3}, \varepsilon n^{1/3}]\!]$,
  \begin{equation}
    \label{eq:397}
   \int_0^1\PP(\fp_n \in \Gamma_{u_{j_1}}^{v_{k_1},0}\cap \Gamma_{u_{j_2}}^{v_{k_2},t})dt=  \frac{1}{|\boxx_n|}\int_0^1\EE[|\Gamma_{u_{j_1}}^{v_{k_1},0}\cap \Gamma_{u_{j_2}}^{v_{k_2},t}\cap \boxx_{ n}|]dt \leq Cn^{-2}\Cov(T_{u_{j_1}}^{v_{k_1}},T_{u_{j_2}}^{v_{k_2}}).
  \end{equation}
\end{lemma}
\begin{proof}
  Since $\fp_n$ is independent of the LPP and is uniformly distributed in $\boxx_n$, by definition, for any fixed $t>0$, we have $\PP(\fp_n \in \Gamma_{u_{j_1}}^{v_{k_1},0})= \frac{1}{|\boxx_n|}\EE[|\Gamma_{u_{j_1}}^{v_{k_1},0}\cap \Gamma_{u_{j_2}}^{v_{k_2},t}\cap \boxx_{ n}|]$ and this immediately implies the first equality. To obtain the inequality in the above, we simply note that for any $t>0$, we have $\EE[|\Gamma_{u_{j_1}}^{v_{k_1},0}\cap \Gamma_{u_{j_2}}^{v_{k_2},t}\cap \boxx_{ n}|]\leq \EE[|\Gamma_{u_{j_1}}^{v_{k_1},0}\cap \Gamma_{u_{j_2}}^{v_{k_2},t}|]=\cO_{j_1,j_2}^{k_1,k_2}(t,n)$ along with Lemma \ref{lem:9}.
\end{proof}

In view of Lemma \ref{lem:7} and the above result, the task now is to obtain precise estimates on $\Cov(T_{u_{j_1}}^{v_{k_1}},T_{u_{j_2}}^{v_{k_2}})$ for different values of $j_1,j_2,k_1,k_2$. For doing so, it will be convenient to introduce some notation-- for $\Delta\in \ZZ$, we define the interval $A_\Delta$ by
  \begin{equation}
    \label{eq:457}
    A_\Delta=
    \begin{cases}
      [\![-\Delta,0]\!], &\textrm{ if } \Delta>0,\\
      [\![0, |\Delta|]\!], &\textrm{ if } \Delta\leq 0.
    \end{cases}
\end{equation}
Now, we state two lemmas without proof and use these to complete the proof of Proposition \ref{prop:1}. Afterwards, we shall provide the proof of these lemmas.

\begin{proposition}
  \label{prop:35}
There exists a constant $C$ such that for all $n$ and all $j,k,i,\Delta\in [\![-2\varepsilon n^{1/3},2\varepsilon n^{1/3}]\!]$ additionally satisfying $i\in A_{\Delta}$, we have %
  \begin{equation}
    \label{eq:399}
    |\Cov(T_{u_j}^{v_{k}}, T_{u_{j+i}}^{v_{k+\Delta + i}})|\leq C(1+|\Delta|)^{-2}n^{2/3}.%
  \end{equation}

\end{proposition}

\begin{proposition}
  \label{prop:34}
There exist constants $C,c_1,c_2$ such that for all $n$ and all $j,k,i,\Delta\in [\![-2\varepsilon n^{1/3},2\varepsilon n^{1/3}]\!]$ additionally satisfying $i\not\in A_{\Delta}$, we have %
  \begin{equation}
    \label{eq:388}
    |\Cov(T_{u_{j}}^{v_k}, T_{u_{j+i}}^{v_{k+\Delta+i}})|\leq Cn^{2/3} \min(e^{-c_1\Delta^2},e^{-c_2\min_{a\in A_{\Delta}}|i-a|^3}).
  \end{equation}
\end{proposition}

\begin{proof}[Proof of Proposition \ref{prop:1} assuming Propositions \ref{prop:35}, \ref{prop:34}]
  In view of Lemma \ref{lem:7} and Lemma \ref{lem:8}, we need only show that for some constant $C$, we have
  \begin{equation}
    \label{eq:6}
    \sum_{|j_1|,|j_2|,|k_1|,|k_2|\leq \varepsilon n^{1/3}}\Cov ( T_{u_{j_1}}^{v_{k_1}},T_{u_{j_2}}^{v_{k_2}})\leq C n^{4/3}\log n.
  \end{equation}
 To prove the above, it suffices to show that there is a constant $C$ such that for any fixed $j_1,k_1$ as in the above, we have
  \begin{displaymath}
    \sum_{|j_2|,|k_2|\leq \varepsilon n^{1/3}}\Cov ( T_{u_{j_1}}^{v_{k_1}},T_{u_{j_2}}^{v_{k_2}})\leq C n^{2/3}\log n.
  \end{displaymath}
  The task now is to use Propositions \ref{prop:35}, \ref{prop:34} to obtain the above inequality. 

  Indeed, we can write
  \begin{align}
    \label{eq:15}
    &\sum_{|j_2|,|k_2|\leq \varepsilon n^{1/3}}\Cov ( T_{u_{j_1}}^{v_{k_1}},T_{u_{j_2}}^{v_{k_2}})\nonumber\\
    &\leq \sum_{|i|,|\Delta|\leq 2\varepsilon n^{1/3}}|\Cov(T_{u_{j_1}}^{v_{k_1}}, T_{u_{j_1+i}}^{v_{k_1+\Delta+i}})|\nonumber\\
    &= \sum_{|\Delta|\leq 2\varepsilon n^{1/3}}\left[\sum_{i\in A_\Delta}|\Cov(T_{u_{j_1}}^{v_{k_1}}, T_{u_{j_1+i}}^{v_{k_1+\Delta+i}})|+\sum_{ i\notin A_\Delta,|i|\leq 2\varepsilon n^{1/3} }|\Cov(T_{u_{j_1}}^{v_{k_1}}, T_{u_{j_1+i}}^{v_{k_1+\Delta+i}})|\right]. %
  \end{align}
  Now, we estimate both the above sums separately. Indeed, by using Proposition \ref{prop:35}, there is a constant $C>0$ such that for all $\Delta\in [\![-2\varepsilon n^{1/3}, 2\varepsilon n^{1/3}]\!]$, we have %
  \begin{equation}
    \label{eq:449}
    \sum_{i\in A_\Delta}|\Cov(T_{u_{j_1}}^{v_{k_1}}, T_{u_{j_1+i}}^{v_{k_1+\Delta+i}})|\leq (|\Delta|+1)\times (Cn^{2/3}(1+|\Delta|)^{-2})= C(1+|\Delta|)^{-1}n^{2/3},
  \end{equation}
  Now, regarding the second term in \eqref{eq:15}, by invoking Proposition \ref{prop:34}, we obtain that for some positive constants $C_1,c_1$,
\begin{equation}
  \label{eq:450}
  \sum_{ i\notin A_\Delta,|i|\leq 2\varepsilon n^{1/3} }|\Cov(T_{u_{j_1}}^{v_{k_1}}, T_{u_{j_1+i}}^{v_{k_1+\Delta+i}})|\leq C_1e^{-c_1\Delta^2}n^{2/3}.
\end{equation}
As a result, \eqref{eq:15} now implies that for some constant $C'$,
\begin{align}
  \label{eq:451}
  \sum_{|j_2|,|k_2|\leq \varepsilon n^{1/3}}\Cov ( T_{u_{j_1}}^{v_{k_1}},T_{u_{j_2}}^{v_{k_2}})\leq   \sum_{|\Delta|\leq 2\varepsilon n^{1/3}}( C(1+|\Delta|)^{-1}n^{2/3}+ C_1e^{-c_1\Delta^2}n^{2/3})\leq C'n^{2/3}\log n,
\end{align}
This completes the proof.
\end{proof}
It now remains to prove Propositions \ref{prop:35}, \ref{prop:34}, and this is the goal of the remainder of the section.
\subsection{Covariance estimate 1: The proof of Proposition \ref{prop:35}}
\label{sec:covar-estim-1}
The now is to prove Proposition \ref{prop:35}. We start by noting that the following simpler result suffices to prove the above proposition.
\begin{proposition}
  \label{prop:36}
  There exists a constant $C$ such that for $j,k,\Delta$ satisfying $|j|,|k|\leq 2\varepsilon n^{1/3}$ and $1\leq |\Delta|\leq 2\varepsilon n^{1/3}$ and all $n$, we have $\Cov(T_{u_j}^{v_k}, T_{u_j}^{v_{k+\Delta}})\leq C\Delta^{-2}n^{2/3}$.
\end{proposition}
\begin{proof}[Proof of Proposition \ref{prop:35} assuming Proposition \ref{prop:36}]
First note that the case $\Delta=0$ follows immediately by the Cauchy-Schwartz inequality, and thus, we can assume $\Delta\neq 0$. Now, we observe that for $i\in A_{\Delta}$, we have $\LL_{u_j}^{v_k}\cap \LL_{u_{j+i}}^{v_{k+\Delta + i}}\neq \emptyset$. As a result, by using the integrability from Proposition \ref{prop:50}, for any $j,k,i,\Delta$ as in the statement of the proposition which additionally satisfy $in^{2/3}\in \ZZ$, we must have $u_{j+i}-u_j=v_{k+\Delta+i}-v_{k+\Delta}$ and therefore must have
  \begin{equation}
    \label{eq:458}
    \Cov(T_{u_j}^{v_{k}}, T_{u_{j+i}}^{v_{k+\Delta + i}})= \Cov(T_{u_j}^{v_k}, T_{u_j}^{v_{k+\Delta}})\leq C\Delta^{-2}n^{2/3},
  \end{equation}
  where the last inequality follows by Proposition \ref{prop:36}.
  Now, in the general case when $in^{2/3}\notin \ZZ$, there is a slight rounding off error since we might not exactly have $u_{j+i}-u_j=v_{k+\Delta+i}-v_{k+\Delta}$. However, we note that $v_{k+\Delta+i}- (u_{j+i}-u_{i})= v_{k+\Delta'}$ for some $\Delta'$ satisfying $|\Delta'-\Delta|\leq 2n^{-2/3}$. Thus by using Proposition \ref{prop:50} along with Proposition \ref{prop:35}, for some constant $C'$, we must have
  \begin{equation}
    \label{eq:652}
    \Cov(T_{u_j}^{v_{k}}, T_{u_{j+i}}^{v_{k+\Delta + i}})= \Cov(T_{u_j}^{v_{k}}, T_{u_j}^{v_{k+\Delta'}})\leq C\Delta'^{-2}n^{2/3}\leq C(\Delta - 2n^{-2/3})^{-2}n^{2/3}\leq C'\Delta^{-2/3}n^{2/3},
  \end{equation}
  where we use $|\Delta|\geq 1$ for the last inequality.
\end{proof}
\begin{figure}
  \centering
  \includegraphics[width=0.7\linewidth]{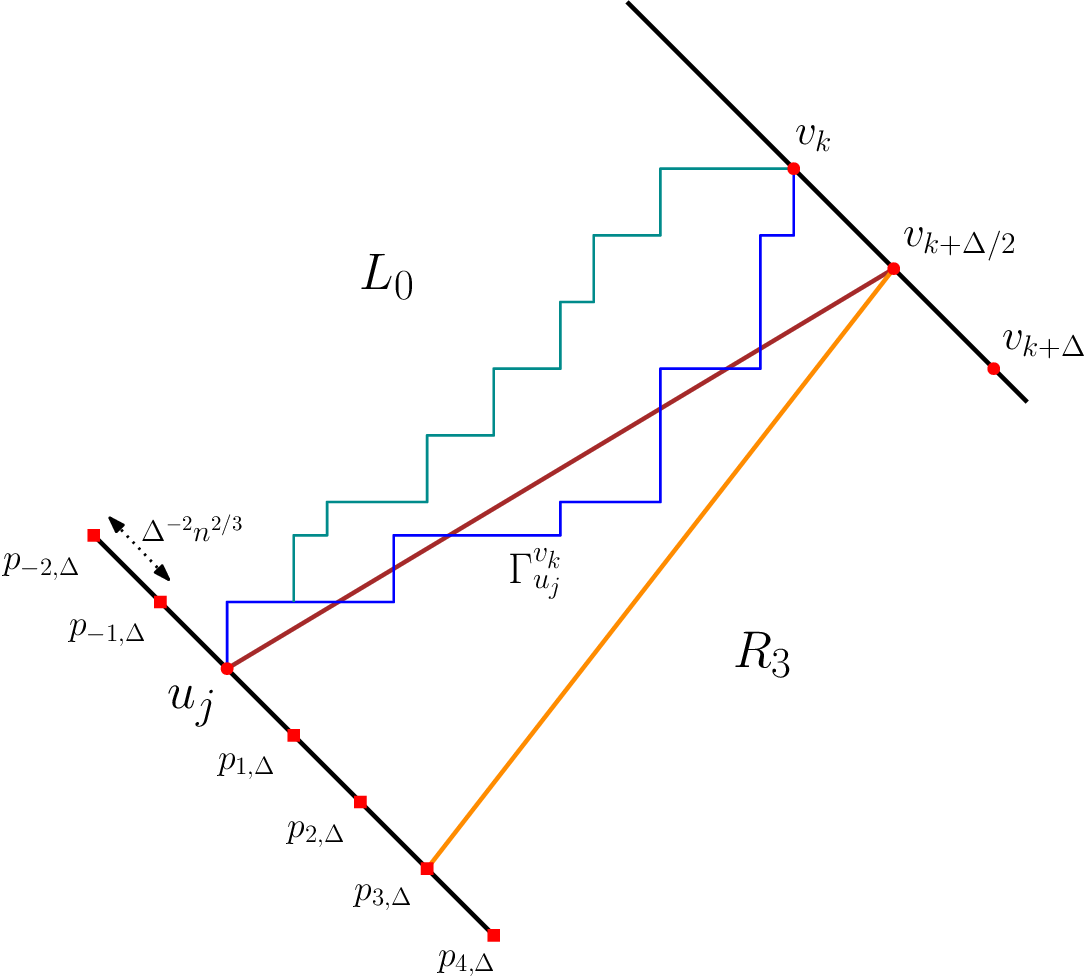}
  \caption{The separation between consecutive points $p_{i,\Delta}$ and $p_{i+1,\Delta}$ is roughly $\Delta^{-2}n^{2/3}$. The region $L_i$ is to the left of the line connecting $p_{i,\Delta}$ and $v_{k+\Delta/2}$ while the region $R_i$ is to the right of it. Here, the region to the left of the brown line is $L_0$ while the region to the right of the orange line is $R_3$. The blue path here is the geodesic $\Gamma_{u_j}^{v_k}$ while the cyan path is one which attains $T_{u_j}^{v_k}\lvert_{L_0}$. Lemma \ref{lem:97} shows stretched exponential tails for $T_{u_j}^{v_k}-T_{u_j}^{v_k}\lvert_{L_0}$ at the scale $\Delta^{-1} n^{1/3}$. }
  \label{fig:ppoints}
\end{figure}
 In order to prepare for the proof of Proposition \ref{prop:36}, we introduce some notation. For points $p,q$ with $-\phi(p)=\phi(q)=n$, we shall use $L_{p,q},R_{p,q}$ to denote the part of $\{z: |\phi(z)|\leq n\}$ strictly to the left and right of $\LL_{p}^{q}$ respectively. That is, we have the disjoint union
\begin{equation}
  \label{eq:398}
  \{z\in \RR^2: |\phi(z)|\leq n\}= L_p^q\cup \LL_{p}^{q}\cup R_p^q.
\end{equation}
Further, we define the points $p_{i,\Delta}=u_{j+i\Delta^{-2}}$. %
For convenience, we now introduce the shorthands:
  \begin{equation}
    \label{eq:402}
    L_{i}= L_{p_{i,\Delta}}^{v_{k+\Delta/2}}, R_{i}=R_{p_{i,\Delta}}^{v_{k+\Delta/2}}.
  \end{equation}
  
  We refer the reader to Figure \ref{fig:ppoints} for a depiction of the objects just defined above. We note that $p_{i,\Delta},L_i,R_i$ also depend on $j,k$ but this dependency is suppressed to avoid clutter. Indeed, for the remainder of this section, we shall simply think of $j,k$ as fixed and satisfying $|j|,|k|\leq 2\varepsilon n^{1/3}$. Now, with the above notation at hand, we can write %
  \begin{align}
    \label{eq:405}
    T_{u_j}^{v_k}&=T_{u_j}^{v_k}\lvert_{L_0}+\sum_{i=0}^{\Delta-1} (T_{u_j}^{v_k}\lvert_{L_{i+1}}- T_{u_j}^{v_k}\lvert_{L_{i}})+ (T_{u_j}^{v_k}-T_{u_j}^{v_k}\lvert_{L_\Delta}),\nonumber\\
    T_{u_j}^{v_{k+\Delta}}&=T_{u_j}^{v_{k+\Delta}}\lvert_{R_0}+\sum_{i=0}^{\Delta-1} (T_{u_j}^{v_{k+\Delta}}\lvert_{R_{-(i+1)}}- T_{u_j}^{v_{k+\Delta}}\lvert_{R_{-i}}) + (T_{u_j}^{v_{k+\Delta}}- T_{u_j}^{v_{k+\Delta}}\lvert_{R_{-\Delta}}).
  \end{align}
  The utility of the above decomposition is that since $L_0$ and $R_0$ are disjoint, the ``main terms'' $T_{u_j}^{v_{k+\Delta}}\lvert_{R_0}$ and $T_{u_j}^{v_k}\lvert_{L_0}$ are almost independent, as we record in the following trivial lemma.
  \begin{lemma}
    \label{lem:145}
    For $\Delta\geq 1$, the random variables $T_{u_j}^{v_k}\lvert_{L_0}-\omega_{u_j}, T_{u_j}^{v_{k+\Delta}}\lvert_{R_0}-\omega_{u_j}$ are measurable with respect to $\{\omega_z\}_{z\in L_0}$ and $\{\omega_z\}_{z\in R_0}$ respectively and are thus independent.
  \end{lemma}
  
In order to use \eqref{eq:405}, it will be important to us that the terms $ T_{u_j}^{v_k}-  T_{u_j}^{v_k}\lvert_{L_0}$ and $T_{u_j}^{v_{k+\Delta}}- T_{u_j}^{v_{k+\Delta}}\lvert_{R_0}$ both be of the right scale $\Delta^{-1}n^{1/3}=(\Delta^{-3}n)^{1/3}$. We now state a lemma achieving the above.
    \begin{lemma}
    \label{lem:97}
There exist constants $C,c$ such that for all $n$ and all $1\leq \Delta\leq 2\varepsilon n^{1/3}$, all $|j|, |k|\leq 3\varepsilon n^{1/3}$, and all $\alpha>0$,
    \begin{align}
      \label{eq:416}
      &\PP(T_{u_j}^{v_k}- T_{u_j}^{v_k}\lvert_{L_0}\geq \alpha \Delta^{-1}n^{1/3})\leq Ce^{-c\sqrt{\alpha}},\nonumber\\
      &\PP(T_{u_j}^{v_{k+\Delta}}- T_{u_j}^{v_{k+\Delta}}\lvert_{R_0}\geq \alpha \Delta^{-1} n^{1/3})\leq Ce^{-c\sqrt{\alpha}}.
    \end{align}
  \end{lemma}
  In order to prove the above lemma, we shall need the following transversal fluctuation estimates involving the $L$ and $R$ sets introduced above.
  \begin{lemma}
    \label{lem:113}
    There exist constants $C,c$ such that for all $n$, and all $\Delta,j,k,i$ satisfying $1\leq \Delta\leq 2\varepsilon n^{1/3}$, $|j|,|k|\leq 3\varepsilon n^{1/3}$ and $0\leq i\leq \Delta^3/4$, we have
    \begin{equation}
      \label{eq:509}
      \PP(\Gamma_{u_j}^{v_k}\subseteq L_i)\geq 1-Ce^{-ci}.
      \end{equation}
      Further, for $\alpha\leq \Delta^3$, we have
    \begin{equation}
      \label{eq:515}
      \PP(\Gamma_{u_j}^{v_k}\cap \{z: \phi(z)\geq -n+ \alpha \Delta^{-3}n\}\subseteq L_0)\geq 1-Ce^{-c\alpha}.
    \end{equation}
  \end{lemma}

  \begin{proof}
We first prove \eqref{eq:509}. For $r\in [\![0,2n]\!]$, Let $z_r$ be the unique point satisfying $z_r\in \Gamma_{u_j}^{v_k}$ and $\phi(z_r)=r-n$ and define $f_i(r)$ by
    \begin{equation}
      \label{eq:511}
      f_i(r)=(1-\frac{r}{2n})(in^{2/3}\Delta^{-2})+ \frac{r}{2n}(\Delta n^{2/3}/2)
    \end{equation}
    Since $i\leq \Delta^{3}/4$, it can be checked that we always have
    \begin{equation}
      \label{eq:512}
      f_i(r)\geq in^{2/3}\Delta^{-2}+ r\Delta n^{-1/3}/16.
    \end{equation}
    Now, by using the mesoscopic transversal fluctuation estimate Proposition \ref{prop:40} along with a union bound, we have
    \begin{align}
      \label{eq:510}
      \PP(\Gamma_{u_j}^{v_k}\not\subseteq L_i)&\leq \sum_{r=0}^{2n}\PP(z_r\notin L_i)\leq \sum_{r=0}^{2n}C\exp(-c\left(\frac{f_i(r)}{r^{2/3}}\right)^3)\nonumber\\
      &\leq \sum_{r=0}^{2n}C\exp(-c(r^{-2/3}in^{2/3}\Delta^{-2}+r^{1/3}\Delta n^{-1/3}/16)^3).
    \end{align}
    Now, note that when seen as a function of $r$, the expression $r^{-2/3}in^{2/3}\Delta^{-2}+r^{1/3}\Delta n^{-1/3}/16$ is minimised when $r= 32in\Delta^{-3}$, and for this value of $r$, it is equal to $c'i^{1/3}$ for an explicit constant $c'$. As a result, by a simple computation involving an exponential series, we obtain that for some constants $C_1,c_1$,
    \begin{equation}
      \label{eq:513}
\sum_{r=0}^{2n}C\exp(-c(r^{-2/3}in^{2/3}\Delta^{-2}+r^{1/3}\Delta n^{-1/3}/16)^3)\leq C_1e^{-c_1 (i^{1/3})^3}=C_1e^{-c_1i}.      
\end{equation}
On combining the above with \eqref{eq:510}, the proof of \eqref{eq:509} is complete.

We now come to the proof of \eqref{eq:515}. Again, by using Proposition \ref{prop:40} along with a union bound, we have
\begin{align}
  \label{eq:516}
  \PP(\Gamma_{u_j}^{v_k}\cap \{z: \phi(z)\geq -n+ \alpha \Delta^{-3}n\}\not\subseteq L_0)&\leq \sum_{r=\lfloor \alpha \Delta^{-3} n\rfloor}^{2n}\PP(z_r\notin L_0)\nonumber\\
                                                                                          & \leq \sum_{r=\lfloor \alpha \Delta^{-3} n\rfloor }^{2n}C_1\exp(-c_1(\frac{r\Delta n^{-1/3}}{r^{2/3}})^3)\nonumber\\
  &\leq Ce^{-c\alpha}.
\end{align}
This completes the proof.
\end{proof}
We now provide the proof of Lemma \ref{lem:97}.
  \begin{proof}[Proof of Lemma \ref{lem:97}]
We only prove the first inequality-- the second one can be obtained by an analogous argument.  We consider the cases $\alpha> \Delta^3$ and $\alpha\leq \Delta^3$ separately. To handle the former case, we simply note that for some constants $C,c$,    \begin{align}
      \label{eq:667}
      \PP(T_{u_j}^{v_k}- T_{u_j}^{v_k}\lvert_{L_0}\geq \alpha \Delta^{-1}n^{1/3})&\leq \PP( T_{u_j}^{v_k}- \EE T_{u_j}^{v_k} \geq \alpha \Delta^{-1}n^{1/3}/2) + \PP(T_{u_j}^{v_k}\lvert_{L_0}- \EE T_{u_j}^{v_k}\leq -\alpha \Delta^{-1} n^{1/3}/2)\nonumber\\
      &\leq Ce^{-c \alpha \Delta^{-1}}\leq Ce^{-c\alpha^{2/3}},
    \end{align}
    where the above uses the moderate deviation estimates in Proposition \ref{prop:42} along with $\alpha >\Delta^{3}$.
    \begin{figure}
      \centering
      \includegraphics[width=0.7\linewidth]{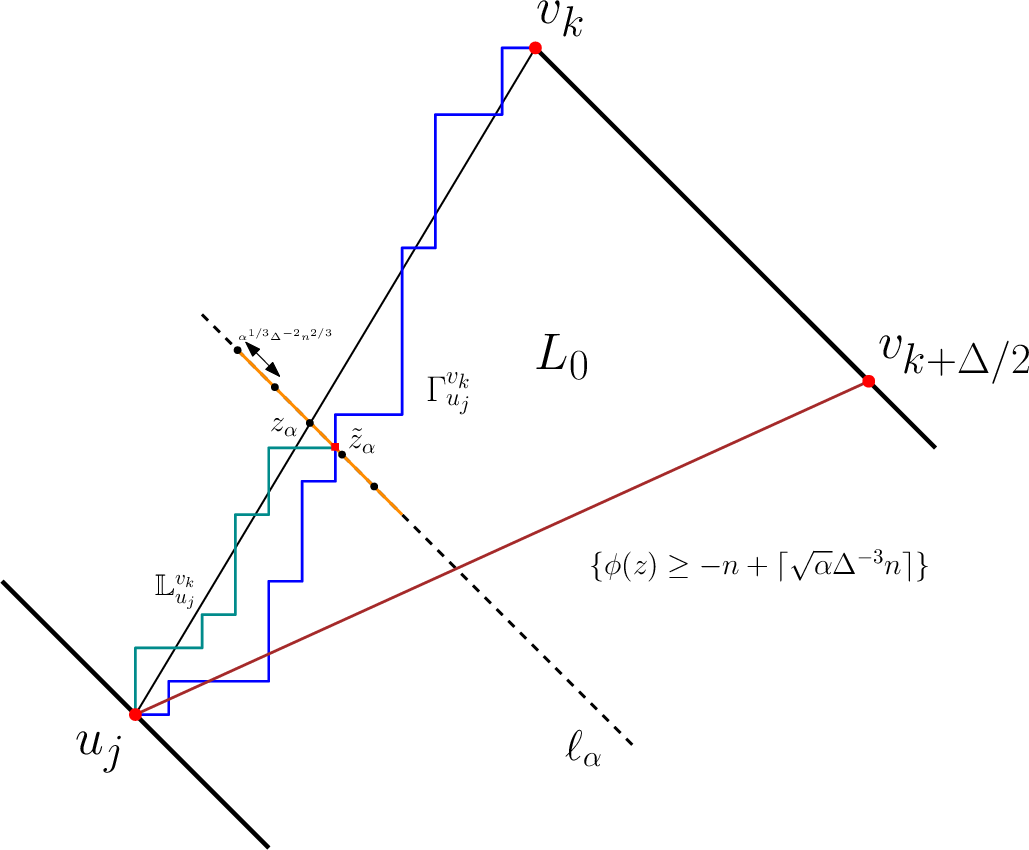}
      \caption{We consider the point $\tilde z_{\alpha}$ where the geodesic $\Gamma_{u_j}^{v_k}$ intersects the line $\ell_\alpha$. While the geodesic $\Gamma_{u_j}^{v_k}$ here does not lie in the region $L_0$, we do have $\Gamma_{u_j}^{v_k}\cap \{\phi(z)\geq -n + \lceil \sqrt{\alpha} \Delta^{-3} n\rceil \}\subseteq L_0$. We now consider a path $\gamma$ (cyan) from $u_j$ to $\tilde z_{\alpha}$ lying within $L_0\cup \{u_j\}$ and attaining $T_{u_j}^{\tilde z_\alpha}\lvert_{L_0}$ and concatenate $\gamma$ with $\Gamma_{\tilde z_{\alpha}}^{v_k}$ to obtain a path from $u_j$ to $v_k$ whose length is within $\alpha \Delta^{-1} n^{1/3}$ of the passage time $T_{u_j}^{v_k}$.
      }
      \label{fig:concat}
    \end{figure}
Now, we consider the case $\alpha\leq \Delta^3$, and we refer the reader to Figure \ref{fig:concat} for a depiction of the argument for this case. Consider the line $\ell_\alpha= \{z:\phi(z)=-n+\lceil\sqrt{\alpha} \Delta^{-3}n\rceil\}$. Let $z_\alpha$ be such that $z_\alpha\in \ell_\alpha\cap \LL_{u_j}^{v_k}$. %
 Let $\tilde z_\alpha$ be the unique point such that $\tilde z_\alpha\in \Gamma_{u_j}^{v_k}\cap \ell_\alpha$. Consider the event $E_\alpha$ defined by
    \begin{equation}
      \label{eq:421}
      E_\alpha=\{|\tilde z_\alpha-z_\alpha|\leq \sqrt{\alpha} \Delta^{-2}n^{2/3}/16\}\cap \{\Gamma_{u_j}^{v_k}\cap \{z: \phi(z)\geq -n+ \sqrt{\alpha} \Delta^{-3}n\}\subseteq L_0\}.
    \end{equation}
    Then by an application of Proposition \ref{prop:40} and Lemma \ref{lem:113}, we have
    \begin{equation}
      \label{eq:419}
      \PP(E_\alpha)\geq 1- Ce^{-c\sqrt{\alpha}}.
    \end{equation}
    By a simple argument involving a concatenation of geodesics, on the event $E_\alpha$, we have
    \begin{equation}
      \label{eq:514}
      T_{u_j}^{v_k}\lvert_{L_0}\geq T_{u_j}^{v_k}- T_{u_j}^{\tilde z_\alpha}+T_{u_j}^{\tilde z_\alpha}\lvert_{L_0},
    \end{equation}
    and as a result, we have $T_{u_j}^{v_k}- T_{u_j}^{v_k}\lvert_{L_0}\leq T_{u_j}^{\tilde z_\alpha}-T_{u_j}^{\tilde z_\alpha}\lvert_{L_0}$. Now, we divide the line segment $\{|z-z_\alpha|\leq \sqrt{\alpha}\Delta^{-2}n^{2/3}/16\}\cap \ell_\alpha\subseteq L_0$ into $O(\alpha^{1/6})$ many line segments $I_i$ of length $\alpha^{1/3}\Delta^{-2}n^{2/3}$ each; in the Figure \ref{fig:concat}, the $I_i$ are the orange line segments connecting the small black dots). By using the estimates from Proposition \ref{prop:42}, we obtain that
    \begin{equation}
      \label{eq:420}
      \PP(\sup_{i} \sup_{z\in I_i}|T_{u_j}^{z}-\EE T_{u_j}^z|\geq \alpha \Delta^{-1} n^{1/3}/2)\leq C\alpha^{1/6} e^{-c\sqrt{\alpha}}.
    \end{equation}
    In fact, by the restricted estimates from Proposition \ref{prop:42}, we also have
    \begin{equation}
      \label{eq:422}
      \PP(\sup_{i} \sup_{z\in I_i}|T_{u_j}^{z}\lvert_{L_0}-\EE T_{u_j}^z|\geq \alpha \Delta^{-1} n^{1/3}/2)\leq C\alpha^{1/6} e^{-c\sqrt{\alpha}}.
    \end{equation}
   We now write
    \begin{align}
      \label{eq:423}
      &\PP(T_{u_j}^{v_k}- T_{u_j}^{v_k}\lvert_{L_0}\geq \alpha \Delta^{-1/3}n^{1/3})\nonumber\\
      &\leq \PP(E_\alpha^c)+ \PP(T_{u_j}^{v_k}- T_{u_j}^{v_k}\lvert_{L_0}\geq \alpha \Delta^{-1/3}n^{1/3}; E_\alpha) \nonumber\\
      &\leq Ce^{-c\sqrt{\alpha}} + \PP(\sup_i\sup_{z\in I_i}  (T_{u_j}^z-T_{u_j}^z\lvert_{L_0})\geq \alpha \Delta^{-1} n^{1/3})\nonumber\\
      &\leq  Ce^{-c\sqrt{\alpha}}+ \PP(\sup_{i} \sup_{z\in I_i}|T_{u_j}^{z}-\EE T_{u_j}^z|\geq \alpha \Delta^{-1} n^{1/3}/2)+ \PP(\sup_{i} \sup_{z\in I_i}|T_{u_j}^{z}\lvert_{L_0}-\EE T_{u_j}^z|\geq \alpha \Delta^{-1} n^{1/3}/2)  \nonumber\\
      &\leq C'e^{-c'\sqrt{\alpha}},
    \end{align}
    where the last inequality is obtained by using \eqref{eq:420} and \eqref{eq:422}.
  \end{proof}
  Later, we shall use \eqref{eq:405} to expand out the covariance $\Cov(T_{u_j}^{v_k}, T_{u_j}^{v_{k+\Delta}})$. While doing so, there shall be a number of ``error'' terms that will come up, and we now prove a few lemmas that will be used to control these.

\begin{lemma}
  \label{lem:95}
  There exist constants $C,c$ such that for all $n$, all $1\leq \Delta\leq 2\varepsilon n^{1/3}$, all $|j|,|k|\leq 3\varepsilon n^{1/3}$, and all $0\leq i\leq \Delta^3/4$, we have
  \begin{align}
    \label{eq:406}
    &\EE(T_{u_j}^{v_k}\lvert_{L_{i+1}}- T_{u_j}^{v_k}\lvert_{L_{i}})^2\leq \EE(T_{u_j}^{v_k}- T_{u_j}^{v_k}\lvert_{L_{i}})^2\leq  Ce^{-ci}\Delta^{-2}n^{2/3},\nonumber\\
    &\EE(T_{u_j}^{v_{k+\Delta}}\lvert_{R_{-(i+1)}}- T_{u_j}^{v_{k+\Delta}}\lvert_{R_{-i}})^2\leq \EE(T_{u_j}^{v_{k+\Delta}}- T_{u_j}^{v_{k+\Delta}}\lvert_{R_{-i}})^2\leq Ce^{-ci}\Delta^{-2}n^{2/3}.
  \end{align}

  \end{lemma}
  \begin{proof}
    We only prove the first equation since the latter can be obtained by a symmetry argument. Now, the first inequality here is immediate by the definition of restricted passage times and thus we need only show that $\EE(T_{u_j}^{v_k}- T_{u_j}^{v_k}\lvert_{L_{i}})^2\leq  Ce^{-ci}\Delta^{-2}n^{2/3}$. By Lemma \ref{lem:113}, for some constants $C,c$, we have
    \begin{equation}
      \label{eq:407}
      \PP(\Gamma_{u_j}^{v_k}\subseteq L_i)\geq 1-Ce^{-ci}.
    \end{equation}
    On the event $\{\Gamma_{u_j}^{v_k}\subseteq L_i\}$, we must have $T_{u_j}^{v_k}=T_{u_j}^{v_k}\lvert_{L_{i}}$. Also, we note that $T_{u_j}^{v_k}- T_{u_j}^{v_k}\lvert_{L_{i}}\geq 0$ a.s. by the definition of restricted passage times. As a result of the above, we have
    \begin{equation}
      \label{eq:408}
      \PP(T_{u_j}^{v_k}- T_{u_j}^{v_k}\lvert_{L_{i}}>0)\leq Ce^{-ci}.
    \end{equation}
    Now, by the Cauchy-Schwartz inequality, for some constants $C',c'$, we have
    \begin{align}
      \label{eq:409}
      \EE(T_{u_j}^{v_k}- T_{u_j}^{v_k}\lvert_{L_{i}})^2&\leq \PP(T_{u_j}^{v_k}\lvert_{L_{i+1}}- T_{u_j}^{v_k}\lvert_{L_{i}}>0)^{1/2} (\EE(T_{u_j}^{v_k}- T_{u_j}^{v_k}\lvert_{L_{i}})^4)^{1/2}\nonumber\\
                                                                       &\leq \sqrt{C}e^{-ci/2}(\EE(T_{u_j}^{v_k}- T_{u_j}^{v_k}\lvert_{L_{i}})^4)^{1/2}\nonumber\\
                                                                       &\leq \sqrt{C}e^{-ci/2}(\EE(T_{u_j}^{v_k}-T_{u_j}^{v_k}\lvert_{L_0})^4)^{1/2}\nonumber\\
      &\leq C'e^{-c' i} \Delta^{-2}n^{2/3}.
    \end{align}
    To obtain the third line above, we have used that $0\leq T_{u_j}^{v_k}- T_{u_j}^{v_k}\lvert_{L_{i}}\leq T_{u_j}^{v_k}-T_{u_j}^{v_k}\lvert_{L_0}$ holds a.s.\ and to obtain the last line, we have used Lemma \ref{lem:97}.

  \end{proof}

     \begin{lemma}
    \label{lem:98}
    There exist constants $C,c$ such that for all $n$, all $1\leq \Delta\leq 2\varepsilon n^{1/3}$, all $|j|,|k|\leq 2\varepsilon n^{1/3}$, and all $0\leq i\leq \Delta^{3}/4$, we have
    \begin{align}
      \label{eq:418}
      &|\Cov( T_{u_j}^{v_k}\lvert_{L_{i+1}}-T_{u_j}^{v_k}\lvert_{L_i}, T_{u_j}^{v_k+\Delta}\lvert_{R_0})|\leq Ce^{-ci}\Delta^{-2}n^{2/3},\nonumber\\
      &|\Cov(T_{u_j}^{v_k}\lvert_{L_0},T_{u_j}^{v_{k+\Delta}}\lvert_{R_{-(i+1)}}- T_{u_j}^{v_{k+\Delta}}\lvert_{R_{-i}})|\leq Ce^{-ci}\Delta^{-2}n^{2/3}.
    \end{align}
  \end{lemma}
  \begin{proof}
By a symmetry argument, it suffices to prove the first inequality. Recall the definition $p_{i,\Delta}=u_{j+i\Delta^{-2}}$. Now, since $R_{i+1}$ and $L_{i+1}$ are disjoint, the vertex weights $\{\omega_{z}\}_{z\in R_{i+1}}\cup\{\omega_{p_{(i+1),\Delta}}\}$ are independent of $\{\omega_{z}\}_{z\in L_{i+1}}$. As a result of this, we have
    \begin{equation}
      \label{eq:424}
      \Cov( T_{u_j}^{v_k}\lvert_{L_{i+1}}-T_{u_j}^{v_k}\lvert_{L_i}, T_{u_j}^{v_k+\Delta}\lvert_{R_0})= \Cov( T_{u_j}^{v_k}\lvert_{L_{i+1}}-T_{u_j}^{v_k}\lvert_{L_i}, T_{u_j}^{v_k+\Delta}\lvert_{R_0}- T_{p_{(i+1),\Delta}}^{v_{k+\Delta}}\lvert_{R_{i+1}}).
    \end{equation}
  Now, note that for some constant $C$, we have the following inequalities:
    \begin{align}
      \label{eq:426}
      &\Var(T_{u_j}^{v_{k+\Delta}}- T_{u_j}^{v_{k+\Delta}}\lvert_{R_0})\leq C\Delta^{-2} n^{2/3},\nonumber\\
      &\Var(T_{p_{(i+1),\Delta}}^{v_{k+\Delta}}- T_{p_{(i+1),\Delta}}^{v_{k+\Delta}}\lvert_{R_{i+1}})\leq C\Delta^{-2} n^{2/3},\nonumber\\
      &\Var(T_{u_j}^{v_{k+\Delta}} - T_{p_{(i+1),\Delta}}^{v_{k+\Delta}})\leq \EE[(T_{u_j}^{v_{k+\Delta}} - T_{p_{(i+1),\Delta}}^{v_{k+\Delta}})^2]\leq C(i+1)\Delta^{-2}n^{2/3}.
    \end{align}
    The first two inequalities above follow by an application of Lemma \ref{lem:97} while the last inequality is an application of Proposition \ref{prop:37}. By using the above inequalities along with the triangle inequality, we immediately obtain
    \begin{equation}
      \label{eq:427}
      \Var(T_{u_j}^{v_k+\Delta}\lvert_{R_0}- T_{p_{(i+1),\Delta}}^{v_{k+\Delta}}\lvert_{R_{i+1}})\leq C(i+3)\Delta^{-2}n^{2/3}.
    \end{equation}
    Finally, by using \eqref{eq:424} along with the Cauchy-Schwartz inequality, we obtain
    \begin{align}
      \label{eq:428}
      |\Cov( T_{u_j}^{v_k}\lvert_{L_{i+1}}-T_{u_j}^{v_k}\lvert_{L_i}, T_{u_j}^{v_k+\Delta}\lvert_{R_0})|&\leq \sqrt{\Var( T_{u_j}^{v_k}\lvert_{L_{i+1}}-T_{u_j}^{v_k}\lvert_{L_i})}\sqrt{ \Var(T_{u_j}^{v_k+\Delta}\lvert_{R_0}- T_{p_{(i+1),\Delta}}^{v_{k+\Delta}}\lvert_{R_{i+1}})}\nonumber\\
                                                                                                        &\leq (C_1e^{-c_1 i}\Delta^{-1}n^{1/3})(\sqrt{C(i+3)}\Delta^{-1}n^{1/3})\nonumber\\
      &\leq \sqrt{C}C_1\sqrt{i+3}e^{-c_1i}\Delta^{-2}n^{2/3}.
    \end{align}
     Here, in the second line, we have used Lemma \ref{lem:95} along with \eqref{eq:427}.
   \end{proof}

We are now finally ready to complete the proof of Proposition \ref{prop:36}.
\begin{proof}[Proof of Proposition \ref{prop:36}]
  By a symmetry argument, it suffices to work with $\Delta>0$. Recall the expansions for $T_{u_j}^{v_k}$ and $T_{u_j}^{v_{k+\Delta}}$ from \eqref{eq:405}. First, we write
    \begin{align}
      \label{eq:429}
      &\Cov(T_{u_j}^{v_k},T_{u_j}^{v_{k+\Delta}})\nonumber\\
      &= \Cov(T_{u_j}^{v_k},T_{u_j}^{v_{k+\Delta}}\lvert_{R_0})+\sum_{i'=0}^{\Delta-1}\Cov(T_{u_j}^{v_k},T_{u_j}^{v_{k+\Delta}}\lvert_{R_{-(i'+1)}}- T_{u_j}^{v_{k+\Delta}}\lvert_{R_{-i'}}) + \Cov(T_{u_j}^{v_k},T_{u_j}^{v_{k+\Delta}}- T_{u_j}^{v_{k+\Delta}}\lvert_{R_{-\Delta}}).
    \end{align}
    By using Lemma \ref{lem:95} along with the Cauchy-Schwartz inequality, we immediately obtain
    \begin{align}
      \label{eq:430}
      |\Cov(T_{u_j}^{v_k},T_{u_j}^{v_{k+\Delta}}- T_{u_j}^{v_{k+\Delta}}\lvert_{R_{-\Delta}})|&\leq \sqrt{\Var(T_{u_j}^{v_k})}\sqrt{\Var(T_{u_j}^{v_{k+\Delta}}- T_{u_j}^{v_{k+\Delta}}\lvert_{R_{-\Delta}})}\nonumber\\
      &\leq Ce^{-c\Delta} \Delta^{-2}n^{2/3},
    \end{align}
    and this bounds the last term in \eqref{eq:429}. To bound the first term therein, we write
    \begin{align}
      \label{eq:431}
      &|\Cov(T_{u_j}^{v_k},T_{u_j}^{v_{k+\Delta}}\lvert_{R_0})|\nonumber\\
      &\leq |\Cov(T_{u_j}^{v_k}\lvert_{L_0},T_{u_j}^{v_{k+\Delta}}\lvert_{R_0})|+\sum_{i=0}^{\Delta-1} |\Cov(T_{u_j}^{v_k}\lvert_{L_{i+1}}- T_{u_j}^{v_k}\lvert_{L_{i}},T_{u_j}^{v_{k+\Delta}}\lvert_{R_0})|+ |\Cov(T_{u_j}^{v_k}-T_{u_j}^{v_k}\lvert_{L_\Delta}, T_{u_j}^{v_{k+\Delta}}\lvert_{R_0})|\nonumber\\
      &\leq \Var(\omega_{u_j}) + \sum_{i=0}^{\Delta-1}C'e^{-c'i}\Delta^{-2}n^{2/3} + Ce^{-c\Delta}\Delta^{-2}n^{2/3}\leq C_1\Delta^{-2}n^{2/3},
    \end{align}
    where to obtain the first term in the third line above, we have used Lemma \ref{lem:145}, and to obtain the next two terms, we have used the Cauchy-Schwartz inequality along with Lemma \ref{lem:98}-- we have also used that $\Var(T_{u_j}^{v_{k+\Delta}}\lvert_{R_0})=O(n^{2/3})$ and this follows from the restricted estimates in Proposition \ref{prop:42}.
It remains to control the sum appearing in \eqref{eq:429}. Locally, defining $A_i=T_{u_j}^{v_{k+\Delta}}\lvert_{R_{-(i'+1)}}- T_{u_j}^{v_{k+\Delta}}\lvert_{R_{-i'}}$, we have for $i'\in [\![0,\Delta-1]\!]$,
    \begin{align}
      \label{eq:432}
      |\Cov(T_{u_j}^{v_k},A_{i'})|&\leq |\Cov(T_{u_j}^{v_k}\lvert_{L_0},A_{i'})|+\sum_{i=0}^{\Delta-1} |\Cov(T_{u_j}^{v_k}\lvert_{L_{i+1}}- T_{u_j}^{v_k}\lvert_{L_{i}},A_{i'})|+ |\Cov(T_{u_j}^{v_k}-T_{u_j}^{v_k}\lvert_{L_\Delta}, A_{i'})|\nonumber\\
      &\leq C_1e^{-c_1i'}\Delta^{-2}n^{2/3} + \sum_{i=0}^{\Delta-1}C_2e^{-c_2(i+i')}\Delta^{-2}n^{2/3}+ C_3e^{-c_3(\Delta+ i')}\Delta^{-2} n^{2/3}\nonumber\\
      &\leq C_4e^{-c_4 i'}\Delta^{-2}n^{2/3},
    \end{align}
    where the first term in the second line is obtained by using Lemma \ref{lem:98} while the other terms in the second line are obtained by using the Cauchy-Schwartz inequality along with Lemma \ref{lem:95}.
    Finally, by using \eqref{eq:429} and combining \eqref{eq:430}, \eqref{eq:431}, \eqref{eq:432}, we obtain
    \begin{align}
      \label{eq:433}
      |\Cov(T_{u_j}^{v_k},T_{u_j}^{v_{k+\Delta}})|\leq C_1\Delta^{-2}n^{2/3}+ \sum_{i'=0}^{\Delta-1} C_2e^{-c_2i'}\Delta^{-2}n^{2/3} + C_3e^{-c_3 \Delta} \Delta^{-2}n^{2/3}\leq C_4\Delta^{-2}n^{2/3},
    \end{align}
    and this completes the proof.
  \end{proof}

\subsection{Covariance estimate 2: The proof of Proposition \ref{prop:34}}
\label{sec:covar-estim-2}
The goal of this section is to prove Proposition \ref{prop:34}. In contrast to the proof of Proposition \ref{prop:35}, the above can be done by an easy transversal fluctuation argument. First, by a straightforward symmetry argument involving replacing the vertex weights $\{\omega_z\}_{z\in \ZZ^2}$ by $\{\omega_{-z}\}_{z\in \ZZ^2}$, we note that in order to prove Proposition \ref{prop:34}, it suffices to prove the following result.
\begin{lemma}
  \label{lem:136}
  There exist constants $C,c_1,c_2$ such that for all $n$ and all $j,k,i,\Delta\in [\![-2\varepsilon n^{1/3},2\varepsilon n^{1/3}]\!]$ additionally satisfying $\Delta\geq 0$ and $i\geq 1$,
  \begin{equation}
    \label{eq:654}
        |\Cov(T_{u_{j}}^{v_k}, T_{u_{j+i}}^{v_{k+\Delta+i}})|\leq Cn^{2/3} \min(e^{-c_1\Delta^2},e^{-c_2i^3}).
  \end{equation}
\end{lemma}
The aim now is to prove the above lemma. We shall require the following lemma controlling the transversal fluctuations of the associated geodesics.
\begin{lemma}
  \label{lem:99}
  For $j,k,i,\Delta\in [\![-2\varepsilon n^{1/3},2\varepsilon n^{1/3}]\!]$ additionally satisfying $\Delta\geq 0$ and $i\geq 1$, consider the event $E_{i,\Delta}$ defined by
  \begin{equation}
    \label{eq:442}
    E_{i,\Delta}=\{\Gamma_{u_j}^{v_k}\subseteq L_{u_{j+i/2}}^{v_{k+(\Delta+i)/2}}, \Gamma_{u_{j+i}}^{v_{k+\Delta+i}}\subseteq R_{u_{j+i/2}}^{v_{k+(\Delta+i)/2}}\}.
  \end{equation}
  Then for some positive constants $C,c_1,c_2$, and all $n$, we have
   \begin{equation}
    \label{eq:440}
    \PP(E_{i,\Delta})\geq 1-C\min(e^{-c_1\Delta^2},e^{-c_2i^3}).
  \end{equation}
\end{lemma}
\begin{proof}
   First, we show that for some constants $C_2,c_2$, we have
   \begin{equation}
     \label{eq:447}
     \PP(E_{i,\Delta})\geq 1-C_2e^{-c_2i^3}.
   \end{equation}
   To do so, we first note that for a constant $c'>0$, using $d(\cdot,\cdot)$ to denote the Euclidean distance, we have
  \begin{equation}
    \label{eq:441}
    d(\LL_{u_j}^{v_k}, \LL_{u_{j+i/2}}^{v_{k+(\Delta+i)/2}}), d(\LL_{u_{j+i}}^{v_{k+(\Delta+i)}}, \LL_{u_{j+i/2}}^{v_{k+(\Delta+i)/2}})\geq c'in^{2/3},
  \end{equation}
  and then we simply apply Proposition \ref{prop:39}. It now remains to show that for some constants $C_1,c_1$, we have $\PP(E_{i,\Delta})\geq 1-C_1e^{-c_1 \Delta^2}$.

  To do so, we first note that by the ordering of geodesics, and since $i\geq 1$, we have the inequalities
  \begin{equation}
    \label{eq:517}
    \PP( \Gamma_{u_j}^{v_k}\subseteq L_{u_{j+i/2}}^{v_{k+(\Delta+i)/2}})\geq \PP(\Gamma_{u_j}^{v_k}\subseteq L_{u_{j+1/4}}^{v_{k+\Delta/2}})= \PP(\Gamma_{u_j}^{v_k}\subseteq L_{\Delta^2/4}),
  \end{equation}
  where we are using the notation from \eqref{eq:402} for the last equality. Thus, by using Lemma \ref{lem:113}, we obtain
  \begin{equation}
    \label{eq:518}
    \PP( \Gamma_{u_j}^{v_k}\subseteq L_{u_{j+i/2}}^{v_{k+(\Delta+i)/2}})\geq 1-C_3e^{-c_3\Delta^2}.
  \end{equation}
  By a symmetry argument, we also obtain that
  \begin{equation}
    \label{eq:519}
    \PP(\Gamma_{u_{j+i}}^{v_{k+\Delta+i}}\subseteq R_{u_{j+i/2}}^{v_{k+(\Delta+i)/2}})\geq 1-C_3e^{-c_3\Delta^2}.
  \end{equation}
  As a result of this, we immediately get $\PP(E_{i,\Delta})\geq 1-2C_3e^{-c_3\Delta^2}$, and this when combined with \eqref{eq:447} completes the proof.
\end{proof}
We now use the above to prove Lemma \ref{lem:136}.

\begin{proof}[Proof of Lemma \ref{lem:136}]
  Just for this proof, we use the shorthand $L=L_{u_{j+i/2}}^{v_{k+(\Delta+i)/2}},R=R_{u_{j+i/2}}^{v_{k+(\Delta+i)/2}}$. Recall the event $E_{i,\Delta}$ from Lemma \ref{lem:99}.%

  Note that on the event $E_{i,\Delta}$, we have
  \begin{equation}
    \label{eq:435}
    T_{u_j}^{v_k}=T_{u_j}^{v_k}\lvert_L, T_{u_{j+i}}^{v_{k+i+\Delta}}=T_{u_{j+i}}^{v_{k+i+\Delta}}\lvert_R.
  \end{equation}
  As a result, if we define
  \begin{equation}
    \label{eq:437}
    A=T_{u_j}^{v_k}-T_{u_j}^{v_k}\lvert_L, B= T_{u_{j+i}}^{v_{k+i+\Delta}}-T_{u_{j+i}}^{v_{k+i+\Delta}}\lvert_R,
  \end{equation}
  then by a simple application of the Cauchy-Schwartz inequality and Lemma \ref{lem:99}, we have
  \begin{equation}
    \label{eq:438}
    (\EE A^2)^{1/2},(\EE B^2)^{1/2}\leq C\min(e^{-c_1\Delta^2},e^{-c_2i^3})n^{1/3}.
  \end{equation}
Note that in the above, we have used that all of $\Var(T_{u_{j}}^{v_k}), \Var(T_{u_j}^{v_k}\lvert_L), \Var(T_{u_{j+i}}^{v_{k+i+\Delta}}), \Var(T_{u_{j+i}}^{v_{k+i+\Delta}}\lvert_R)$ are $O(n^{2/3})$, and this is a consequence of Proposition \ref{prop:42}. Thus, we now have
  \begin{align}
    \label{eq:436}
    &|\Cov(T_{u_{j}}^{v_k}, T_{u_{j+i}}^{v_{k+\Delta+i}})|\nonumber\\
    &\leq  |\Cov(T_{u_{j}}^{v_k}\lvert_L, T_{u_{j+i}}^{v_{k+i+\Delta}}\lvert_R)|+ |\Cov(T_{u_{j}}^{v_k}\lvert_L, B)| + |\Cov(A, T_{u_{j+i}}^{v_{k+i+\Delta}}\lvert_R)|+ |\Cov(A,B)|\nonumber\\
    &\leq C\min(e^{-c_1\Delta^2},e^{-c_2i^3})n^{2/3},
  \end{align}
  where we have used that the first term is zero due to independence, and the remaining three terms are controlled by the Cauchy-Schwartz inequality.
\end{proof}

\subsection{Completion of the proof of Theorem \ref{thm:4}}
The goal now is to combine the moment estimates Propositions \ref{prop:2}, \ref{prop:1} to complete the proof of Theorem \ref{thm:4}. Recall the set $\boxx_n$ and random variables $X_n$ defined in \eqref{eq:376}, \eqref{eq:1}. First, we note that by the second moment method, we immediately have the following result. 
\label{sec:obta-exist}
\begin{lemma}
  \label{lem:2.1}
  There exists a positive constant $C$ such that $\PP(X_n>0)\geq C(\log n)^{-1}$ for all $n$.
\end{lemma}
\begin{proof}
  By the second moment method, we know that $\PP(X_n>0)\geq \EE(X_n)^2/(\EE X_n)^2\geq C(\log n)^{-1}$, where we have used Proposition \ref{prop:2} and Proposition \ref{prop:1} to obtain the above inequality.
\end{proof}
By using Lemma \ref{lem:2.1} along with the pigeonhole principle, we immediately have the following.
\begin{lemma}
  \label{lem:10}
There is a positive constant $C$ such that for each $n\in \NN$, there is a deterministic point $\fq_n\in \boxx_{n}$ for which
  \begin{displaymath}
    \PP(\fq_n\in \Gamma_{u_j}^{v_k,t} \textrm{ for some } j,k\in [\![-\varepsilon n^{1/3}, \varepsilon n^{1/3}]\!] \textrm{ and for some } t\in (0,1) %
    )\geq C(\log n)^{-1}.
  \end{displaymath}
\end{lemma}
We are now ready to complete the proof of Theorem  \ref{thm:4}. %

\begin{proof}[Proof of Theorem \ref{thm:4}]

  Locally, for points $p\leq q\in \ZZ^2$, let us consider the quantity
  \begin{equation}
    \label{eq:656}
    \slope^*(p,q)= \frac{\psi(q)-\psi(p)}{\phi(q)-\phi(p)}.
  \end{equation}
 Note that for all $j,k\in [\![-\varepsilon n^{1/3}, \varepsilon n^{1/3}]\!]$, we have
  \begin{equation}
    \label{eq:655}
    \slope^*(u_j,v_k)\in (\theta-2\varepsilon,\theta+2\varepsilon)
  \end{equation}
  for all $n$ large enough. Now, for a set $S\subseteq \RR$, $\fluc^S_n$ denote the event that there exists a $t\in [0,1]$, $j,k\in [\![-\varepsilon n^{1/3}, \varepsilon n^{1/3}]\!]$ and two points $p^t,q^t\in \Gamma_{u_j}^{v_k,t}$ satisfying $\phi(q^t)-\phi(p^t)\geq n/10$ along with
  \begin{equation}
    \label{eq:669}
    |\slope^*(p^t,q^t)-\slope^*(u_j,v_k)|\geq \varepsilon.
  \end{equation}
  Note that by using transversal fluctuation estimates for static exponential LPP, there exist constants $C,c$ such that for any fixed $t\in \RR$, we immediately have%
  \begin{equation}
    \label{eq:448}
    \PP(\fluc_n^{\{t\}})\leq Ce^{-cn}.
  \end{equation}
  The goal now is to bound the probability of $\fluc^{[0,1]}_n$. Locally for two points $u\leq v\in \ZZ^2$, we use $\sqr(u,v)$ to denote the lattice square with two of its diagonal endpoints as $u,v$. First, let $\cT_n^{[0,1]}\subseteq [0,1]$ denote the set of times $t$ at which there is at least some vertex $z\in \bigcup_{j,k} \sqr(u_j,v_k)$ whose weight is resampled at time $t$, where the union is over $j,k\in [\![-\varepsilon n^{1/3}, \varepsilon n^{1/3}]\!]$. It is easy to see that $|\bigcup_{j,k} \sqr(u_j,v_k)|\leq Cn^2$ for some constant $C$. Now, since
  \begin{equation}
    \label{eq:668}
    \cT_n^{[0,1]}\sim \mathrm{Poi}\left(|\bigcup_{j,k} \sqr(u_j,v_k)|\right),
  \end{equation}
and by using that for some positive constants $C',C_1,c_1$, the estimate $\PP(\textrm{Poi}(Cn^2)\geq C'n^2)\leq C_1e^{-c_1n^2}$ holds, we have
  \begin{equation}
    \label{eq:425}
    \PP(|\cT_n^{[0,1]}|\geq C'n^{2})\leq C_1e^{-c_1n^2}.
  \end{equation}
  \begin{figure}
    \centering
    \includegraphics[width=0.7\linewidth]{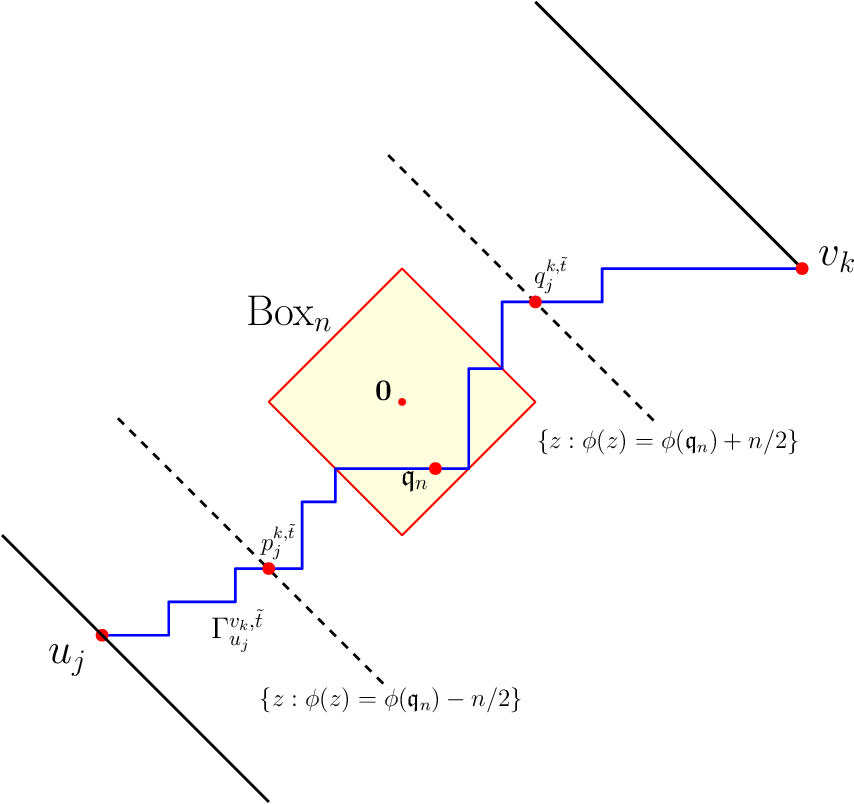}
    \caption{Here, $\mathfrak{q}_n\in \boxx_n$ is a deterministic point with at least a $C(\log n)^{-1}$ probability of there being a geodesic $\Gamma_{u_j}^{v_k,\tilde t}$ going via it for some $|j|,|k|\leq \varepsilon n^{1/3}$ and $\tilde t\in [0,1]$. Additionally, on the event $\fluc_n^{[0,1]}$, we must additionally have $p_j^{k,\tilde t}\in \mathfrak{q}_n+\ell_{-n/2,3\varepsilon}^\theta$ and $q_j^{k,\tilde t}\in \mathfrak{q}_n+\ell_{n/2,3\varepsilon}^\theta$.}
    \label{fig:lbcomplete}
  \end{figure}
  Now, we note that conditional on $\cT_n^{[0,1]}$, for any fixed $t\in \cT_n^{[0,1]}$, the environment $\omega^t$ is simply static exponential LPP. Also, note that by definition, on the set $[0,1]\setminus (\cT_n^{[0,1]})^c$, all the geodesics that we are interested in stay unchanged. Indeed, it can be seen that a.s.\ for every $t\in [0,1]$, there exists a $t'\in \{0\}\cup \cT_n^{[0,1]}$ for which we have $\Gamma_{u_j}^{v_k,t}=\Gamma_{u_{j}}^{v_k,t'}$ for all $j,k\in [\![-\varepsilon n^{1/3}, \varepsilon n^{1/3}]\!]$. As a result, we can write
\begin{align}
  \label{eq:443}
  \PP(\fluc_n^{[0,1]})&\leq \PP(|\cT_n^{[0,1]}|\geq C'n^2)+ \EE[\ind(|\cT_n^{[0,1]}|\leq C'n^2)\sum_{t\in \{0\}\cup \cT_n^{[0,1]}}\PP(\fluc_n^{\{t\}}\lvert \cT_n^{[0,1]})]\nonumber\\
  &\leq C_1e^{-c_1n^2}+ (C'n^2)(Ce^{-cn})\leq C_2e^{-c_2n},
\end{align}
where we have used \eqref{eq:448} to obtain the last line above. Now, we consider the event $F_n$ given by
\begin{equation}
  \label{eq:452}
  F_n=\{\fq_n\in \Gamma_{u_j}^{v_k,\tilde t} \textrm{ for some } j,k \in [\![-\varepsilon n^{1/3}, \varepsilon n^{1/3}]\!]\textrm{ and for some } \tilde t\in (0,1)\}\cap \fluc_n^{[0,1]},
\end{equation}
and as a consequence of \eqref{eq:443} and Lemma \ref{lem:10}, we immediately have that for some constant $C$,
\begin{equation}
  \label{eq:460}
  \PP(F_n)\geq C(\log n)^{-1}.
\end{equation}
Now, if we define $p_{j}^{k,t}, q_{j}^{k,t}\in \Gamma_{u_j}^{v_k,t}$ to be the unique points (see Figure \ref{fig:lbcomplete}) which additionally satisfy $\phi(p_j^{k,t}-\fq_n)=-n/2$, $\phi(q_j^{k,t}-\fq_n)=n/2$, then by using \eqref{eq:655} and the condition \eqref{eq:669} in the definition of $\fluc_n^{[0,1]}$, we obtain that on the event $F_n$, we must have $\psi(\fq_n-p_j^{k,\tilde t}),\psi(q_j^{k,\tilde t}-\fq_n)\in ((\theta-3\varepsilon) n/2,(\theta+3\varepsilon) n/2)$. Further, we emphasize that, by definition, we have $\fq_n\in \Gamma_{p_j^{k,\tilde t}}^{q_j^{k,\tilde t},\tilde t}$. 

As a result of this, in the notation of Theorem \ref{thm:4}, we have
\begin{equation}
  \label{eq:454}
  F_n\subseteq \{\exists \tilde t\in [0,1] \textrm{ and points }  p\in \fq_n+\ell^\theta_{-n/2,3\varepsilon}, q\in \fq_n+\ell^\theta_{n/2,3\varepsilon} \textrm{ with } \fq_n\in \Gamma_{p}^{q,\tilde t}\}.
\end{equation}
Finally, by using the above, we can write
\begin{align}
  \label{eq:459}
  &\PP(\exists \tilde t\in [0,1]\textrm{ and points } p\in \ell_{-n/2,3\varepsilon}^\theta, q\in \ell_{n/2,3\varepsilon}^\theta \textrm { with } \0\in \Gamma_{p}^{q,\tilde t})\nonumber\\
  &=\PP(\exists \tilde t\in [0,1] \textrm{ and points }  p\in \fq_n+\ell_{-n/2,3\varepsilon}^\theta, q\in \fq_n+\ell_{n/2,3\varepsilon}^\theta \textrm{ with } \fq_n\in \Gamma_{p}^{q,\tilde t})\geq C(\log n)^{-1},
\end{align}
where the first inequality follows by the translation invariance of exponential LPP and the second inequality follows by \eqref{eq:454} and \eqref{eq:460}. Replacing $\varepsilon$ by $\varepsilon/3$ and $n$ by $2n$ now completes the proof.

\end{proof}

\printbibliography
\end{document}